\documentclass{article}

\usepackage[a4paper,top=2.54cm,bottom=2.54cm,left=3.17cm,right=3.17cm,%
            includehead,includefoot]{geometry}

\usepackage{amssymb}
\usepackage{amsmath}
\usepackage{amsfonts,amsthm}
\usepackage{graphicx}
\usepackage{subfig}
\usepackage{float}
\usepackage{tabu}
\usepackage{xcolor}
\usepackage{yhmath}
\usepackage{hyperref}
  \hypersetup{colorlinks,citecolor=blue,linkcolor=blue,breaklinks=true}
\usepackage{epstopdf}
\usepackage{lipsum}
\usepackage{commutative-diagrams}
\usepackage{tikz-cd}
\usepackage{enumitem}

\allowdisplaybreaks

\theoremstyle{plain}
\newtheorem{theorem}{Theorem}[section]
\newtheorem{lemma}[theorem]{Lemma}
\newtheorem{corollary}[theorem]{Corollary}
\newtheorem{proposition}[theorem]{Proposition}

\theoremstyle{definition}
\newtheorem{definition}[theorem]{Definition}

\newtheorem{question}{Question}
\theoremstyle{remark}
\newtheorem{remark}[theorem]{Remark}

\newcommand{\R}{\mathbb{R}}\newcommand{\C}{\mathbf{C}}

\newcommand{\bbm}{\begin{bmatrix}}
\newcommand{\ebm}{\end{bmatrix}}

\begin{document}
\title{\uppercase{The $K$-theory of maximal and reduced Roe algebras for Hecke pairs with equivariant coarse embeddings}}

\author{
  Liang Guo \and Hang Wang \and  Xiufeng Yao
}
\maketitle

\begin{abstract}
In this paper, we generalize the Dirac-dual-Dirac method to Hecke pairs with equivariant coarse embeddings and establish the $K$-theoretic isomorphisms between the maximal and reduced equivariant Roe algebras. We also extend these results to prove the Baum--Connes conjecture in this context.
\end{abstract}

\section{Introduction} 


The Baum-Connes conjecture plays a central role in noncommutative geometry by providing a geometric approach to computing the $K$-theory of group $C^*$-algebras. It claims that for any countable discrete group $\Gamma$ and any $\Gamma$-$C^*$--algebra $B$, the assembly map 
\begin{equation}\label{eq: assembly map}
\mu: RK_*^{\Gamma}(\underline{E}\Gamma, B)\to K_*(B\rtimes_r \Gamma)
\end{equation}
is an isomorphism. The domain of this map, known as the equivariant representable $K$-homology of the classifying space for proper actions $\underline{E}\Gamma$, encodes the topological and geometric information of the group action. Elements of this group are represented by $\Gamma$-equivariant abstract elliptic differential operators on $\underline{E}\Gamma$. The assembly map $\mu$ acts as an index map, which assigns to each such elliptic operator $D$ an element $\mu(D)$ in $K_*(B\rtimes_r \Gamma)$, called the \emph{higher index} of $D$.

However, besides the $K$-theory of the reduced crossed product, there is another very natural choice for the target of the assembly map: the \emph{maximal crossed product} $B \rtimes \Gamma$. This leads to a factorization of the assembly map in \eqref{eq: assembly map}:
$$\mu: RK_*^{\Gamma}(\underline{E}\Gamma, B)\xrightarrow{\mu_{\max}} K_*(B\rtimes \Gamma)\xrightarrow{\lambda_*} K_*(B\rtimes_r \Gamma).$$
where $\mu_{\max}$ is the \emph{maximal} assembly map and $\lambda: B\rtimes \Gamma\to B\rtimes_r\Gamma$ is the canonical quotient map. It is then natural to ask how these two choices relate to each other. In particular, under what conditions do they yield the same $K$-theory?

\begin{question}\label{question1}
Let $\Gamma$ be a countable discrete group. When does the canonical quotient map induce an isomorphism in $K$-theory $\lambda_*: K_*(B\rtimes \Gamma) \to K_*(B\rtimes_{r}\Gamma)$ for any $\Gamma$-$C^*$-algebra $B$?
\end{question}

For amenable groups, the canonical quotient map $\lambda: B\rtimes \Gamma \to B\rtimes_r\Gamma$ is an isomorphism for any $\Gamma$-$C^*$-algebra $B$, which naturally induces an isomorphism in $K$-theory. However, this is false for non-amenable groups (e.g., free groups). Remarkably, J.~Cuntz \cite{Cuntz-1983} showed that for free groups, the map $\lambda$ still induces an isomorphism in $K$-theory, despite not being an isomorphism at the level of $C^*$-algebras. Motivated by this, Cuntz introduced K-amenability for discrete groups as a sufficient condition ensuring Question \ref{question1}, i.e., $\lambda_*$ is an isomorphism. Subsequently, in \cite{Julg-Valette-1985}, P.~Julg and A.~Valette extended the notion of $K$-amenability to general locally compact groups. Inspired by this, we say that a group is \emph{quasi $K$-amenable}  if it yields a positive answer to Question \ref{question1}.

Kasparov's Dirac-dual-Dirac method is a primary tool for studying quasi $K$-amenability. Originally introduced by Kasparov \cite{Kasparov-1988} in the context of  the Novikov and the Baum-Connes conjectures, this method has been extensively developed and now stands as one of the most effective approaches to the Baum-Connes conjecture. 

A major milestone was achieved by Higson and Kasparov \cite{Higson-Kasparov-2001}, who applied this technique to establish the Baum--Connes conjecture with coefficients together with their $K$-amenability for a-T-menable groups. Building on this framework, Yu \cite{Yu-2000} made a breakthrough by adapting the machinery to the coarse geometric setting, proving the Novikov conjecture for groups coarsely embeddable into Hilbert spaces. This work significantly expanded the applicability of the method beyond the realm of amenable groups. 

Subsequent research has extended these results to broader geometric contexts. In the realm of Banach space geometry, Kasparov and Yu \cite{Kasparov-Yu-2012} generalized the coarse embedding theorem to Banach spaces with property (H).  Separately, Lafforgue \cite{Lafforgue-2002, Lafforgue-2012} introduced Banach $KK$-theory, a breakthrough that enabled the treatment of groups with property (T) which were previously inaccessible. Beyond the Banach space setting, Kasparov and Skandalis  \cite{KasparovSkandalisbolic}  verified the conjecture for groups acting on weakly bolic spaces. More recently, Brodzki, Guentner, and Higson \cite{Brodzki-Guentner-Higson-2019}  employed this method to provide an elegant geometric characterization of $K$-amenability for groups acting on CAT(0) cubical spaces, yielding a constructive proof of the Baum--Connes conjecture in this context.

The examples above illustrate a strong parallelism between the study of quasi $K$-amenability and the Baum-Connes conjecture. Within the context of the Baum-Connes conjecture, a natural and fruitful perspective is to investigate its stability under group extensions.
Consider an extension $1 \to N \to G \to G/N \to 1$ of countable discrete groups. H.~Oyono-Oyono \cite{Oyono-2001} showed that if $G/N$ and every subgroup of $G$ containing $N$ with finite index satisfy the Baum-Connes conjecture with coefficients, then so does $G$. While such stability results typically rely on the assembly map being an isomorphism, and generally fail for mere injectivity (the Novikov conjecture), J.~Deng \cite{MR4422220} proved a result in the coarse setting: he proved that if both $N$ and $G/N$ admit coarse embeddings into a Hilbert space, then $G$ satisfies the strong Novikov conjecture. In the coarse setting, J.~Deng and L.~Guo \cite{deng-guo-2025} introduced twisted Roe algebras as coarse counterparts of crossed products. They used this construction to prove a similar stability result. Regarding classical $K$-amenability, Cuntz \cite{Cuntz-1983} showed that this property is preserved under direct products, free products, subgroups, and extensions by amenable normal subgroups. 

Inspired by these results, in this paper, we investigate the stability of quasi $K$-amenability under group extensions. In fact, we consider a more general setting. Recently, C.~Dell'Aiera \cite{Clement-2023} generalized the result of Deng \cite{MR4422220} to the context of Hecke pairs. Recall that a pair $(\Gamma, \Lambda)$ is called a \emph{Hecke pair} if $\Lambda$ is an almost normal subgroup of $\Gamma$ (see Definition \ref{defheckepair}). C.~Dell'Aiera investigated the relationship between the coarse geometry of a Hecke pair and the analytic property of its Schlichting completion. By adapting the techniques of H.~Oyono-Oyono, he proved that for a Hecke pair $(\Gamma, \Lambda)$, if both $\Lambda$ and the quotient space $X = \Gamma/\Lambda$ admit coarse embeddings into a Hilbert space, then the Novikov conjecture holds for $\Gamma$.

In this paper, we investigate the stability of quasi $K$-amenability for Hecke pairs. Our main theorem is stated as follows:

\begin{theorem}\label{A}
Let $(\Gamma, \Lambda)$ be a Hecke pair of discrete groups and let $X = \Gamma/\Lambda$. If $X$ admits a $\Gamma$-equivariant coarse embedding into a Hilbert space and $\Lambda$ is a-T-menable, then $\Gamma$ is quasi $K$-amenable, i.e., for any $\Gamma$-$C^*$-algebra $B$, the canonical quotient map induces an isomorphism in $K$-theory:
\[
K_*(B \rtimes \Gamma) \cong K_*(B \rtimes_r \Gamma).
\]

In particular, assume that $\Lambda\trianglelefteq\Gamma$ is a normal subgroup. If $\Lambda$ and $\Gamma/\Lambda$ are both a-T-menable, then $\Gamma$ is quasi $K$-amenable.
\end{theorem}

The class of a-T-menable groups is known to be closed under extensions by amenable \emph{quotients}. However, as noted in \cite[Corollary 1.4.13]{BHV-2008}, this stability does not hold for extensions by amenable \emph{kernels}. A standard counterexample is the group pair $(SL(2, \mathbb{Z}) \ltimes \mathbb{Z}^2, \mathbb{Z}^2)$, which has relative Property (T) and thus fails to be a-T-menable.
We observe that while Property (T) is an obstruction to quasi $K$-amenability due to the existence of the Kazhdan projection  (as shown by Julg and Valette in \cite[Corollary 3.7]{Julg-Valette-1985}). In this context, Theorem \ref{A} implies that extensions of a-T-menable groups are quasi $K$-amenable, even in cases where the group fails to be a-T-menable due to relative Property (T).

Although both this work and that of Dell'Aiera \cite{Clement-2023} investigate higher index theory for Hecke pairs, our methods are fundamentally different. Dell'Aiera follows the strategy of Oyono-Oyono \cite{Oyono-2001}, employing the Schlichting completion to study the Baum-Connes conjecture. In contrast, our approach relies on the Dirac-dual-Dirac method. A key feature of our approach is a geometric ``cutting-and-pasting'' argument that reduces the global problem for $\Gamma$ to subgroups \emph{commensurable} with $\Lambda$, see Definition \ref{defheckepair} and Corollary \ref{keycoro}.

Once the problem is reduced to subgroups commensurable with $\Lambda$, Theorem~\ref{A} can be proved easily.  Recall that for any subgroup $G \leq \Gamma$ commensurable with $\Lambda$, the intersection $H = G \cap \Lambda$ has finite index in both $G$ and $\Lambda$. Since the a-T-menability is preserved under taking subgroups and finite extensions, it is invariant under commensurability. Therefore, if $\Lambda$ is a-T-menable, then $G$ is also a-T-menable. By the celebrated result of Higson and Kasparov \cite{Higson-Kasparov-2001}, a-T-menability of $G$ implies its $K$-amenability (and hence quasi $K$-amenability). This concludes the proof.

The Dirac-dual-Dirac framework employed in the proof of Theorem \ref{A} can be further adapted to address (coarse) Baum-Connes and Novikov conjectures for the group $\Gamma$. By incorporating different geometric assumptions on $\Lambda$ into this common thread, we obtain the following corollaries for Hecke pairs:

\begin{corollary}\label{B}
Let $(\Gamma, \Lambda)$ be a Hecke pair of discrete groups and $X = \Gamma/\Lambda$ the quotient space.\begin{itemize}
\item (Theorem \ref{BCC}) If $\Lambda$ is a-T-menable and $X$ admits a $\Gamma$-equivariant coarse embedding into Hilbert space, then $\Gamma$ satisfies the Baum--Connes conjecture with coefficients.
\item (Theorem \ref{strongNovikov}) If $\Lambda$ admits a coarse embedding into Hilbert space and $X$ admits a $\Gamma$-equivariant coarse embedding into Hilbert space, then $\Gamma$ satisfies the strong Novikov conjecture with coefficients.
\item (Theorem \ref{coarseBCC}) If both $\Lambda$ and $\Gamma$ admit coarse embeddings into Hilbert space, then $\Gamma$ satisfies the coarse Baum--Connes conjecture and consequently the Novikov conjecture.
\end{itemize}\end{corollary}



The paper is organized as follows. Section 2 introduces Hecke pairs and their coarse geometric properties, and establishes the commensurability condition. Section 3 develops the necessary algebraic framework, defining the Roe algebras and twisted Roe algebras under both reduced and maximal norms. Section 4 presents the proof of the main theorem by adapting Yu's Dirac-dual-Dirac method to the setting of Hecke pairs. Section 5 applies Theorem \ref{A} to verify the Baum--Connes and Novikov conjectures for Hecke pairs, yielding Corollaries \ref{B}. 

\section{Preliminaries on Hecke Pairs}

In this section, we recall the definition and coarse geometric properties of Hecke pairs. The key observation is that when the quotient space $X$ of a Hecke pair $(\Gamma,\Lambda)$ is $\Gamma$-equivariantly coarsely embeddable into Hilbert space, one can construct a $\Gamma$-equivariant coarse embedding $h:X \to \mathcal{H}$ such that for every $v\in \mathcal{H}$, the stabilizer $\Gamma_{v}$ is commensurable with $\Lambda$.

Throughout this paper, we always assume that $\Gamma$ is a countable discrete group with a proper length function $| \cdot |$. We begin with the definition of Hecke pairs.
\begin{definition}\label{defheckepair}
  Let $\Lambda$ be a subgroup of a group $\Gamma$. We say that $\Lambda$ is almost normal in $\Gamma$ if one of the following equivalent conditions holds:
  \begin{itemize}
    \item For any $\gamma \in \Gamma$, the subgroups $\Lambda$ and $\Lambda^{\gamma}:=\gamma \Lambda \gamma^{-1}$ are commensurable. That is, the indices $\left[ \Lambda: \Lambda \cap \gamma\Lambda\gamma^{-1} \right] $ and $\left[ \gamma\Lambda\gamma^{-1}: \Lambda \cap \gamma\Lambda\gamma^{-1} \right] $ are finite.
    \item The left action of $\Lambda$ on the coset space $\Gamma /\Lambda$ has finite orbits, i.e., every double coset $\Lambda s \Lambda$ is a finite union of left cosets $\gamma\Lambda$.
  \end{itemize}
  In this case, the pair $(\Gamma,\Lambda)$ is called a Hecke pair.
\end{definition}

Here, we provide a proof of the equivalence between the two conditions in Definition \ref{defheckepair}.

\begin{proof}[Proof of Equivalence] 
  First, assume that for every $\gamma \in \Gamma$, $\Lambda$ and $\gamma \Lambda \gamma^{-1}$ are commensurable. We will show that the left action of $\Lambda$ on $\Gamma / \Lambda$ has finite orbits.
  
  Fix an arbitrary $s \in \Lambda$ and consider the double coset $\Lambda s \Lambda$, which can be expressed as a union of left cosets: 
  $$\Lambda s \Lambda=\bigcup_{\lambda \in  \Lambda}\lambda s \Lambda .$$ 
  However, many of these cosets coincide. Specifically, for $\lambda_{1},\lambda_{2} \in \Lambda$, we have $\lambda_{1}s\Lambda=\lambda_{2}s\Lambda$ if and only if $  \lambda_{2}^{-1}\lambda_{1}\in \Lambda\cap s\Lambda s^{-1}$. This means that $\lambda_{1}$ and $\lambda_{2}$ are in the same coset of $\Lambda$ modulo the subgroup $\Lambda \cap s \Lambda s^{-1}$.
 
  By assumption, $\Lambda \cap s \Lambda s ^{-1}$ has finite index in $\Lambda$. So, $\Lambda$ can be decomposed into finitely many cosets:$$
  \Lambda=\bigcup_{i=1}^{n} \lambda_{i}(\Lambda \cap s \Lambda s ^{-1}),$$ 
  where $\lambda_1, \dots, \lambda_n$ are representatives of the cosets. Since $\lambda s \Lambda = \lambda' s \Lambda$ whenever $\lambda$ and $\lambda'$ are in the same coset, it follows that 
  $$\Lambda s \Lambda =\bigcup_{i=1}^{n} \lambda_{i}(\Lambda \cap s \Lambda s ^{-1}) (s\Lambda)= \bigcup_{i=1}^{n} \lambda_{i}s \Lambda . $$
  Therefore, $\Lambda s \Lambda$ is a finite union of left cosets, which shows that the orbit of $s \Lambda$ under the left action of $\Lambda$ is finite.

  Conversely, assume that the left action of $\Lambda$ on $\Gamma / \Lambda$ has finite orbits. We will show that for every $s \in \Gamma$, $\Lambda$ and $s \Lambda s^{-1}$ are commensurable.
  
  For any $s \in \Gamma$, the double coset can be written as
  $$\Lambda s \Lambda=\bigcup_{i=1}^{n}\gamma_{i}\Lambda.$$
  As before, for $\lambda_{1} s \Lambda=\lambda_{2} s \Lambda$, we have $\lambda_{1}$ and $\lambda_{2}$ in the same coset of $\Lambda$ modulo $\Lambda \cap s \Lambda s ^{-1}$. This implies that $\Lambda \cap s \Lambda s ^{-1}$ is the stabilizer of $s\Lambda$ under the left action of $\Lambda$. Thus, the number of distinct cosets in $\Lambda s \Lambda$ is equal to the index:
  $$|\Lambda s \Lambda / \Lambda |=| \Lambda / \Lambda\cap s \Lambda s ^{-1}|.$$
  Since $\Lambda s \Lambda / \Lambda$ is finite by assumption, the index $[\Lambda : \Lambda \cap s \Lambda s^{-1}]$ is finite. By symmetry, $[s \Lambda s^{-1} : \Lambda \cap s \Lambda s^{-1}]$ is also finite. Hence, $\Lambda$ and $s \Lambda s^{-1}$ are commensurable.
\end{proof}

The quotient space $X=\Gamma /\Lambda$ carries a natural $\Gamma$-invariant metric defined by 
\begin{equation}\label{X-metric}
  d(s\Lambda,t\Lambda)=\inf_{\lambda,\lambda' \in \Lambda} |\lambda s^{-1} t \lambda' |.
\end{equation}
With respect to this metric, we denote by $B(x,r)$ the open ball of radius $r$ centered at a point $x \in X$. We say that the metric space $(X,d)$ has bounded geometry if for every $r>0$, the cardinality $|B(x,r)|$ is uniformly finite, i.e.,
$$\sup_{x \in X} |B(x,r)|<\infty.$$
It turns out $X$ has bounded geometry by the following proposition. \cite[Proposition 2.1]{Clement-2023}.

\begin{proposition}[{\cite[Proposition 2.1]{Clement-2023}}]\label{boundedgeometry}
The quotient space $X= \Gamma /\Lambda$ is of bounded geometry with the canonical quotient metric if and only if $(\Gamma, \Lambda)$ is a Hecke pair.\qed
\end{proposition}

Now we introduce the definition of equivariant coarse embedding.
\begin{definition}
    A $\Gamma$-space $X$ admits a $\Gamma$-equivariant coarse embedding $h:X \to \mathcal{H}$ into the Hilbert space $\mathcal{H}$ if:
    \begin{itemize}
        \item there is an affine isometric action of $\Gamma$ on $\mathcal{H}$ such that 
        $$
        \gamma\cdot h(x) = h(\gamma\cdot x), \quad \text{for all } \gamma \in \Gamma, x\in X;
        $$
        \item there exists $\rho_{-},\rho_{+}: [0,+\infty) \to [0,+\infty)$ satisfying $\lim\limits_{r \to +\infty}\rho_{-}(r) = +\infty$ such that
        $$
        \rho_{-}(d(x,y)) \le \|h(x)-h(y)\| \le \rho_{+}(d(x,y)).
        $$
    \end{itemize}
\end{definition}

Recall that for a general Hecke pair $(\Gamma, \Lambda)$, Tzanev \cite{MR2015025} constructed an essentially unique new Hecke pair called the Schlichting completion. This construction and its generalizations \cite{MR2465930} provide a standard tool for studying Hecke pairs. In this paper, we adopt a different strategy. Instead of using the completion, we rely on the following key geometric property of the equivariant coarse embedding.

\begin{proposition}\label{Keyprop}
  Given a Hecke pair $(\Gamma,\Lambda)$, if the quotient space $X=\Gamma /\Lambda$ admits a $\Gamma$-equivariant coarse embedding into Hilbert space $\mathcal H$, then for any $v \in \mathcal{H}$ and any $R>0$, the subset 
  $$
  N_{v,R}= \{ \gamma \in \Gamma \mid B(v,R) \cap \gamma\cdot B(v,R) \neq \emptyset \} \subseteq \Gamma
  $$
  has finite orbit on $e\Lambda$ in $X$, i.e., $N_{v,R}\cdot \Lambda$ is finite in $X$, where $e\Lambda$ is the representative of the group identity $e$ in $\Gamma /\Lambda$.
\end{proposition}
\begin{proof}
  Suppose $g:X \to \mathcal{H}_{0}$ is a $\Gamma$-equivariant coarse embedding. Let $\mathcal{V}_{\mathrm{aff}}$ be the affine subspace of $\mathcal{H}_{0}$ spanned by the image $g(X)$. Since the action of $\Gamma$ on $\mathcal{H}_{0}$ is by affine isometries, the subspace $\mathcal{V}_{\mathrm{aff}}$ is $\Gamma$-invariant.
  
  Fix a base point $x_{0} \in X$, and let $w_{0}:=g(x_{0})$. The translation $v \mapsto v - w_{0}$ maps the affine subspace $\mathcal{V}_{\mathrm{aff}}$ onto a linear subspace of $\mathcal{H}_{0}$. Let $\mathcal{H}$ be the closure of this linear subspace. Then $\mathcal{H}$ is a Hilbert space. 

  The action on $\mathcal{H}_0$ induces a conjugate action on $\mathcal{H}$ by
  \begin{equation}
    \beta_\gamma(v) = \alpha_{\gamma}(v+w_{0})-w_{0}, \quad \text{ for all } v \in \mathcal{H} \subset \mathcal{H}_{0}.
  \end{equation}
where $\alpha$ is the original action on $\mathcal{H}_{0}$. Specifically, the action $\beta_\gamma$ on $\mathcal{H}$ is given by
  $$
  \beta_\gamma(v) = U_{\gamma}v + (\alpha_\gamma(w_{0}) - w_{0}),\quad \text{ for all } v \in \mathcal{H}.
  $$
One can easily check tht that $b'_{\gamma}=\alpha_{\gamma}(w_{0}) - w_{0}$ is 1-cocycle with respect to the unitary representation $U$. Therefore, the map $\beta:\Gamma \to \mathrm{Isom}(\mathcal{H})$ defined by $\beta_{\gamma}(v)=U_{\gamma}v + b'_{\gamma}$ gives an affine isometric action of $\Gamma$ on $\mathcal{H}$.  We define a new $\Gamma$-equivariant coarse embedding $h:X \to \mathcal{H}$ defined by
  $$
  h(x) = g(x) - w_{0}, \quad \text{ for all } x \in X.
  $$

  For any $v \in \mathcal{H}$, consider the condition $\gamma \in N_{v,R}$, which implies that $\|\gamma \cdot v - v\| < 2R$. Fix a base point $x_0 = e\Lambda$ in $X$, we have
  $$
  \begin{aligned}
    \|\gamma \cdot h(x_0) - h(x_0)\| &= \|\left(\gamma \cdot h(x_0) - \gamma \cdot v\right) + \left(\gamma \cdot v - v\right) + \left(v - h(x_0)\right)\| \\
    &\leq \|\gamma \cdot \left(h(x_0) - v\right)\| + \|\gamma \cdot v - v\| + \|v - h(x_0)\| \\
    &= \|h(x_0) - v\| + \|\gamma \cdot v - v\| + \|v - h(x_0)\| \\
    &\leq 2\|h(x_0) - v\| + 2R.
  \end{aligned}
  $$
  Let $R' = 2\|h(x_0) - v\| + 2R$. Since $h$ is a coarse embedding, there exists a proper function $\rho_-: \mathbb{R}_+ \to \mathbb{R}_+$ such that
  $$
  \rho_-(d(\gamma \cdot x_0, x_0)) \leq \|h(\gamma \cdot x_0) - h(x_0)\| \leq R'.
  $$
  This implies that $d(\gamma \cdot x_0, x_0) < C$ for some constant $C$. By the bounded geometry of $X$ (Proposition 2.2), the ball $B(x_0, C)$ contains only finitely many points. Therefore, the set $\{\gamma \Lambda | \gamma \in N_{v,R}\} \subset B(x_0, C)$ is finite, which proves the proposition.
\end{proof}

We remark that the restriction to the subspace $\mathcal{H} \subseteq \mathcal{H}_0$ is not strictly necessary for the proof; the proposition remains valid for the entire ambient space $\mathcal{H}_0$. We restrict our attention to $\mathcal{H}$ here primarily to capture the essential geometric part of the action.

The Proposition \ref{Keyprop} has an analogous characterization established in \cite[Theorem 4.1]{Clement-2023}. Their approach relies on the technical machinery of Schlichting completions, deriving properties of the original group from the analytic properties of its completion. In contrast, the commensurability condition established in the Proposition \ref{Keyprop} makes it unnecessary to pass to the completion, providing a more direct perspective. The following corollary plays a crucial role in the proof of our main theorem.

\begin{corollary}\label{keycoro}
  For a Hecke pair $(\Gamma, \Lambda)$, if the quotient space $X$ admits a $\Gamma$-equivariant coarse embedding into Hilbert space, then there exists a $\Gamma$-equivariant coarse embedding $h:X \to \mathcal{H}$, such that for any $v \in \mathcal{H}$, the stabilizer $\Gamma_{v}$ is commensurable with $\Lambda$.
\end{corollary}
\begin{proof}
  Let $v=\sum_{i=1}^{n}a_{i}h(s_{i}\Lambda)$ be an arbitrary vector in the dense affine subspace spanned by $h(X)$, where $\sum_{i=1}^{n}a_{i}=1$. Let $R=0$, then $N_{v,R}$ is the stabilizer $\Gamma_{v}$ of $v$. 
  
  By Proposition \ref{Keyprop}, $\Gamma_{v}$ has finite orbit on $\Lambda$ in $X$. Since 
  $$
  \Gamma_{v}\cdot \Lambda \cong \Gamma_{v}/(\Gamma_{v}\cap \Lambda),
  $$
  we have $[\Gamma_{v}:\Gamma_{v}\cap \Lambda]$ is finite.
  
  On the other hand, consider the finite intersection $K=\bigcap_{i=1}^{n} \Gamma_{s_{i}\Lambda}$. Since $\Gamma_{s_{i}\Lambda}$ is commensurable with $\Lambda$ for each $i$, $K$ is also commensurable with $\Lambda$. Thus, $[\Lambda: K \cap \Lambda]< \infty$.

  For any $\gamma \in K$, $\gamma$ fixes each $s_{i}\Lambda$. Since the action is affine and $\sum_{i=1}^{n} a_i = 1$, we have
  $$
  \gamma \cdot v = \sum_{i=1}^{n} a_i (\gamma \cdot h(s_i \Lambda)) = \sum_{i=1}^{n} a_i h(s_i \Lambda) = v.
  $$
  Thus $K \subset \Gamma_{v}$. It follows that $K \cap \Lambda \subset \Gamma_{v} \cap \Lambda$, and hence 
  $$
  [\Lambda: \Gamma_{v}\cap \Lambda] \le [\Lambda: K \cap \Lambda] < \infty.
  $$
  Combining both sides, we conclude that $\Gamma_{v}$ is commensurable with $\Lambda$.
\end{proof}

\section{Equivariant Roe algebra}

In this section, we introduce the definitions of different equivariant Roe algebras that will be used in the proof of our main theorem. We first recall the equivariant Roe algebra (together with its maximal counterpart) for a metric space with cocompact and proper action of $\Gamma$, whose $K$-theory is isomorphic to that of the reduced (maximal) group $C^{*}$-algebra of $\Gamma$.  Using coarse embeddings, we introduce a twisted version of the equivariant Roe algebra, which serves as a generalization of the crossed product.

\subsection{Roe algebra}\label{section3.1}

Let $\Gamma$ be a countable discrete group acting properly and cocompactly on a proper metric space $(\Delta,d)$. Let $B$ be a $\Gamma$-$C^{*}$-algebra. 

\begin{definition}[\cite{Kasparov-Yu-2012}]\label{usual-prop-spt}
  Let $E_{B}$ be a countably generated Hilbert module over $B$ and let $\pi:C_{0}(\Delta)\to \mathcal{B}(E_{B})$ be a $*$-representation, where $\mathcal{B}(E_{B})$ means the $C^{*}$-algebra of all adjointable operators on $E_{B}$. Let $T \in \mathcal{B}(E_{B})$.
  \begin{enumerate}
    \item The support of $T$, denoted by $\operatorname{supp}(T)$, is defined as the complement of the set of all $(x,y) \in \Delta\times \Delta$ such that there exist $f,g \in C_{0}(\Delta)$ with $f(x)\neq 0$, $g(y)\neq 0$ and $\pi(f)T\pi(g)=0$;
    \item The propagation of $T$, denoted by $\operatorname{prop}(T)$, is defined as $\operatorname{prop}(T)=\sup\{ d(x,y) \mid (x,y) \in \operatorname{supp}(T) \} $;
    \item $T$ is said to be locally compact if for all $f \in C_{0}(\Delta)$, both $\pi(f)T$ and $T\pi(f)$ belong to $\mathcal{K}(E_{B})$, the algebra of compact operators on $E_{B}$.
  \end{enumerate}
\end{definition}

Let $U:\Gamma \to U(E_{B})$ be a unitary representation. We say that the triple $(C_{0}(\Delta),\Gamma,\pi)$ is a covariant system if for any $f \in C_{0}(\Delta)$ and $\gamma \in \Gamma$, 
$$
U_{\gamma}\pi(f)U_{\gamma^{-1}}=\pi(\gamma f), \quad \text{ where } \gamma f(x)=f(\gamma^{-1}x).
$$

\begin{definition}[\cite{Kasparov-Yu-2012}]
  A covariant system $(C_{0}(\Delta),\Gamma,\pi)$ is called to be admissible if:
  \begin{enumerate}
    \item $E_{B}$ is isomorphic to $H_{\Delta}\otimes E \otimes B$ as $\Gamma$-Hilbert modules over $B$, where $H_{\Delta}$ is a $\Gamma$-Hilbert space and $E$ is a separable infinite-dimensional $\Gamma$-Hilbert space;
    \item $\pi=\pi_{0} \otimes 1$ for some $\Gamma$-equivariant $*$-homomorphism $\pi_{0}: C_{0}(\Delta) \to B(H_{\Delta})$, such that $\pi_{0}(f)$ is not in $\mathcal{K}(H_{\Delta})$ for any non-zero function $f \in C_{0}(\Delta)$, and $\pi_{0}$ is non-degenerate in the sense that $\pi_{0}(C_{0}(\Delta))H_{\Delta}$ is dense in $H_{\Delta}$;
    \item For each $x\in \Delta$, $E$ is isomorphic to $\ell^{2}(\Gamma_{x})\otimes H_{x}$ as $\Gamma_{x}$-Hilbert spaces for some Hilbert spaces $H_{x}$ with a trivial $\Gamma_{x}$ action, where $\Gamma_{x}$ is the finite isotropy subgroup of $\Gamma$ at $x$.
  \end{enumerate}
\end{definition}

Such an admissible covariant system always exists. For example, choose 
$$
E_{B}=\ell^{2}(X) \otimes H_{0} \otimes \ell^{2}(\Gamma) \otimes B,
$$
where $X$ is a countable dense $\Gamma$-invariant subset of $\Delta$. The action of $\Gamma$ on $E_B$ is diagonal, given by
$$
\gamma(\xi \otimes \eta \otimes \delta_g \otimes b) = (\gamma \cdot \xi) \otimes \eta \otimes \delta_{\gamma g} \otimes b,
$$
where $(\gamma \cdot \xi)(x) = \xi(\gamma^{-1}x)$ for $\xi \in \ell^2(X)$, and $H_0$ is an infinite-dimensional separable Hilbert space with the trivial $\Gamma$-action. The representation $\pi$ is defined by pointwise multiplication on $\ell^{2}(X)$. Then $E$ in the above definition can be taken as $\ell^{2}(\Gamma)$ with the left regular representation of $\Gamma$. 


\begin{definition}[\cite{Kasparov-Yu-2012}]
  Let $(C_{0}(\Delta),\Gamma,\pi)$ be an admissible covariant system. The algebraic equivariant Roe algebra $\mathbb{C}[\Delta,B]^{\Gamma}$ is defined as the algebra of all the $\Gamma$-invariant, locally compact operators in $\mathcal{B}(E_{B})$ with finite propagation.
\end{definition}

The following definition and proposition are due to 3.13 in \cite{gong-wang-yu-maxroe}.
\begin{definition}
  The reduced equivariant Roe algebra $C^{*}(\Delta,B)^{\Gamma}$ is defined as the norm closure of $\mathbb{C}[\Delta,B]^{\Gamma}$ in $\mathcal{B}(E_{B})$.
\end{definition}

\begin{proposition}\label{reduced-crossed-product}
  If $(C_{0}(\Delta),\Gamma, \pi)$ is an admissible covariant system, then the reduced Roe algebra $C^{*}(\Delta,B)^{\Gamma}$ is isomorphic to the reduced crossed product $(B\rtimes_{r}\Gamma) \otimes \mathcal{K}(E)$.
\end{proposition}

The reader is referred to \cite[Theorem 5.3.2]{willett2020higher} for the proof.

Given $s>0$, denote by $\mathbb{C}_{s}[\Delta, B]^{\Gamma}$ the collection of all operators in $\mathbb{C}[\Delta,B]^{\Gamma}$ whose propagation is less than $s$. The following lemma provides uniform norm estimates, which are essential for defining the maximum Roe algebra.
\begin{lemma}[{\cite[Lemma 3.4]{gong-wang-yu-maxroe}}]\label{normcontrol}
 Suppose $\Delta$ has bounded geometry. For any positive real number $s$, there exists a constant $C_{s}$ such that for every $*$-representation $\phi: \mathbb{C}[\Delta,B]^{\Gamma} \to \mathcal{B}(H_{\phi})$ on some Hilbert space $H_{\phi}$ and every $T \in \mathbb{C}_{s}[\Delta,B]^{\Gamma}$,  
 $$
 \|\phi_{T}\|_{\mathcal{B}(H_{\phi})} \le C_{s}\| T \|_{\mathcal{B}(E_{B})}.
 $$
\end{lemma}

We now give the definition.
\begin{definition}
  The maximal equivariant Roe algebra, denoted by $C^{*}_{\max}(\Delta,B)^{\Gamma}$, is the completion of $\mathbb{C}[\Delta,B]^{\Gamma}$ with respect to the $*$-norm 
  $$
  \|T\|_{\max}:=\sup_{\phi,H_{\phi}} \|\phi(T)\|_{\mathcal{B}(H_{\phi})},
  $$ 
  where $(\phi, H_{\phi})$ runs through representations $\phi$ of $\mathbb{C}[\Delta,B]^{\Gamma}$ on a Hilbert space $H_{\phi}$.
\end{definition}

The following proposition is analogous to the reduced case discussed in \cite{gong-wang-yu-maxroe} (3.14).
\begin{proposition}\label{maximal-crossed-product}
  Let $\Delta$ be a metric space with bounded geometry carrying a cocompact and proper action of $\Gamma$, then the maximal Roe algebra $C^{*}_{\max}(\Delta,B)^{\Gamma}$ is isomorphic to the maximal crossed product $(B\rtimes  \Gamma) \otimes \mathcal{K}(E)$.
\end{proposition}

The universality of the maximal completion gives rise to the canonical quotient map:
$$
\lambda: C^{*}_{\max}(\Delta,B)^{\Gamma} \to C^{*}(\Delta,B)^{\Gamma}.
$$
This map induces a K-theoretic morphism of abelian groups:
$$
\lambda_{*}: K_{*}(C^{*}_{\max}(\Delta,B)^{\Gamma}) \to K_{*}(C^{*}(\Delta,B)^{\Gamma}).
$$

Equip $\Gamma$ with a proper length function $|\cdot |$. Then $\Gamma$ can be regarded as a discrete metric space with bounded geometry under the left-invariant word-length metric. By the Švarc–Milnor Lemma, every space with a proper and cocompact action of $\Gamma$ is $\Gamma$-equivariantly coarsely equivalent to $\Gamma$ itself. Consequently, the equivariant Roe algebras defined above are Morita equivalent to those defined on $\Gamma$. In what follows, we shall therefore set $\Delta = \Gamma$, in which case $\lambda$ (and the induced $\lambda_*$) identifies with the canonical quotient map from the maximal to the reduced group $C^*$-algebra:
$$
\begin{gathered}
 \lambda:C^{*}_{\max}(\Gamma,B)^{\Gamma} \to C^{*}(\Gamma,B)^{\Gamma},\\
 \lambda_{*}: K_{*}(C^{*}_{\max}(\Gamma,B)^{\Gamma}) \to K_{*}(C^{*}(\Gamma,B)^{\Gamma}).
\end{gathered}
$$


\subsection{Twisted Roe Algebra}\label{section3.2}

In this section, we shall introduce the twisted Roe algebras. Let us begin by constructing some necessary geometric framework. 
  
Let $\mathcal{H}$ be the Hilbert space into which the quotient space $X=\Gamma /\Lambda$ admits a $\Gamma$-equivariant coarse embedding, i.e., $h:X \to \mathcal{H}$, as in Proposition \ref{Keyprop}. The group $\Gamma$ acts on $\mathcal{H}$ by affine isometries, and $h$ induces a $\Gamma$-equivariant map 
$$
j: \Gamma \to \mathcal{H}, \quad j(\gamma):=h(\gamma \Lambda), \quad \text{for all } \gamma \in \Gamma.
$$
Note that for any $\lambda \in \Lambda$, we have $j(\lambda)=h(e\Lambda)=j(e)$. Hence $j$ is not a coarse embedding in general.

We now construct the algebra $\mathcal{A}(\mathcal{H})$, following the approach outlined in \cite{Yu-2000} and \cite{Higson-Kasparov-2001}. Let $\mathcal{H}$ be a real Hilbert space. Let $V \subset \mathcal{H}$ be an $n$-dimensional affine subspace of the form
$$V = v_0 + \text{span}\{v_1, v_2, \dots, v_n\},$$
where $v_0, v_1, \dots, v_n \in \mathcal{H}$ and let $V^{0}=V-v_{0}$ be its underlying linear space. Denote by $\mathrm{Cl}(V^{0})$ the complexified Clifford algebra of $V^{0}$. Set 
$$
\mathcal{C}(V):=C_{0}(V,\mathrm{Cl}(V^{0})), \quad \mathcal{S}:=C_{0}(\mathbb{R}), \quad \mathcal{A}(V):=\mathcal{S}\hat{\otimes} \mathcal{C}(V).
$$
Both $\mathcal{C}(V)$ and $\mathcal{S}$ carry natural $\mathbb{Z}_2$-gradings: $\mathcal{C}(V)$ inherits its grading from the even-odd decomposition of $\mathrm{Cl}(V^{0})$, and $\mathcal{S}$ is graded by the parity of functions. Consequently, $\mathcal{A}(V)$ is understood as their graded $C^*$-tensor product.

To pass to the infinite-dimensional Hilbert space $\mathcal{H}$, we choose a directed system $\{ V_{a} \}$  of finite-dimensional affine subspaces such that $\bigcup_{a}V_{a} $ is dense in $\mathcal{H}$. Suppose $V_{a} \subset V_{b}$; then we can write $V_{b}=V_{ba}^{0} \oplus V_{a}$ with a linear subspace $V_{ba}^{0}$, and every $v_{b} \in V_{b}$ admits a unique decomposition $v_{b}=v_{ba}+v_{a}$, where $v_{ba}\in V_{ba}^{0}$ and $v_{a}\in V_{a}$. Any $g \in \mathcal{C}(V_{a})$ extends to be an element of $\mathcal{C}(V_{b})$ by 
$$
g_{ba}(v_{b})=g_{ba}(v_{ba}+v_{a}):=g(v_{a}).
$$ 
Since $V_{ba}^{0} \subset \mathrm{Cl}(V_{b}^{0})$, we then have a unbounded function
$$
C_{ba}: V_{b} \to \mathrm{Cl}(V_{b}^{0}), \quad C_{ba}(v_{b})=v_{ba},
$$
Using the unbounded multiplier $X$ on $\mathcal{S}$ defined by $X(f)(t)=t f(t)$ for $t \in \mathbb{R}$, we extend an element $f\hat{\otimes}g \in \mathcal{A}(V_{a})$ to $\mathcal{A}(V_{b})$ by
\begin{equation}\label{extend}
  \beta_{ba}(f\hat{\otimes}g)=f(X\hat{\otimes}1+1\hat{\otimes}C_{ba})(1\hat{\otimes}g_{ba}). 
\end{equation}

Taking the direct limit over the directed system $\{ V_{a} \}$, we define the $C^{*}$-algebra 
$$
\mathcal{A}(\mathcal{H})=\varinjlim\mathcal{A}(V_{a})=\varinjlim\mathcal{S}\hat{\otimes}\mathcal{C}(V_{a}).
$$
From the discussion in \cite{higson-kasparov-trout-1998}, up to a canonical isomorphism, this algebra is independent of the particular choice of a directed system. Moreover, the algebra $\mathcal{A}(\mathcal{H})$ admits an action of $\Gamma$ induced from that of $\mathcal{H}$.

Equip $\mathbb{R}_{+}\times \mathcal{H}$ with the weakest topology such that the projection to $\mathcal{H}$ is weakly continuous and the function $t^{2}+\|h\|^{2}$ is continuous. Concretely, this topology can be constructed as follows: a net $\{ (t_{i},v_{i}) \}_{i}$ in $\mathbb{R}_{+}\times \mathcal{H}$ converges to $(t,v)$ if and only if $t_{i}^{2}+\|v_{i}\|^{2} \to t^{2}+\|v\|^{2}$, and for any $u\in \mathcal{H}$, $\left<v_{i},u \right> \to \left<v,u \right>$. Under this topology, $\mathbb{R}_{+}\times \mathcal{H}$ is second countable, locally compact, and Hausdorff.

For each finite dimensional subspace $V_a\subseteq \mathcal H$, the commutative algebra $C_{0}(\mathbb{R}_{+}\times V_{a})$ is in the center of $\mathcal{A}(V_{a})$. Moreover, the structure maps $\beta_{ba}$ restrict to homomorphisms from $C_{0}(\mathbb{R}_{+}\times V_{a})$ to $C_{0}(\mathbb{R}_{+}\times V_{b})$. It follows that the inductive limit is isomorphic to $C_{0}(\mathbb{R}_{+}\times \mathcal{H})$, where $\mathbb{R}_{+}\times \mathcal{H}$ is equipped with the topology given above. This makes $\mathcal A(\mathcal H)$ a $(\mathbb{R}_{+}\times \mathcal{H})$-$C^*$-algebra.

Fix a separable complex infinite-dimensional Hilbert space $H_{0}$ on which $\Gamma$ acts trivially, and denote by $\mathcal{K}$ the algebra of compact operators on $H_{0}$. For any $\Gamma$-$C^{*}$-algebra $B$, $\Gamma$ acts diagonally on the tensor product $B\hat{\otimes}\mathcal{K}$. 

\begin{definition}
  For any $a\in \mathcal{A}(\mathcal{H})\hat{\otimes}B\hat{\otimes}\mathcal{K}$, define its support $\operatorname{supp}(a)$ as the complement of the set of all points $v\in \mathbb{R}_{+}\times \mathcal{H}$ for which there exists $f\in C_{0}(\mathbb{R}_{+}\times \mathcal{H})$ with $f(v) \neq 0$ and 
  $$
  (f\hat{\otimes}k)\cdot a=0 \quad \text{ for all } k\in B\hat{\otimes}\mathcal{K}.
  $$ 
\end{definition}


With these notions, we can now define the algebraic twisted Roe algebra. Define
$$
E_{\mathcal{A}}:=\ell^{2}(\Gamma)\hat{\otimes}\mathcal{A}(\mathcal{H})\hat{\otimes}B\hat{\otimes}\mathcal{K}.
$$ 
The space $E_{\mathcal{A}}$ can be viewed as the completion of the space of finitely supported functions $\xi: \Gamma \to \mathcal{A}(\mathcal{H})\hat{\otimes}B\hat{\otimes}\mathcal{K}$ with respect to the $\mathcal{A}(\mathcal{H})\hat{\otimes}B\hat{\otimes}\mathcal{K}$-valued inner product
$$
\left<\xi,\eta \right>:=\sum_{\gamma \in \Gamma}\xi(\gamma)^{*}\eta(\gamma).
$$ 
The right $\mathcal{A}(\mathcal{H})\hat{\otimes}B\hat{\otimes}\mathcal{K}$-module action is given by the pointwise multiplication:
$$
(\xi \cdot a)(\gamma):=\xi(\gamma)a, \quad \text{for all } a \in \mathcal{A}(\mathcal{H})\hat{\otimes}B\hat{\otimes}\mathcal{K}.
$$ 
Moreover, $\Gamma$ acts on $E_{\mathcal{A}}$ by 
$$
(\gamma \cdot \xi)(g)=\gamma\cdot (\xi(\gamma^{-1}g)), \quad \text{for all } \gamma, g \in \Gamma.
$$ 
Thus, $E_{\mathcal{A}}$ carries the structure of a $\Gamma$-Hilbert $C^{*}$-module over $\mathcal{A}(\mathcal{H})\hat{\otimes}B\hat{\otimes}\mathcal{K}$. 

\begin{definition}
  \label{defalgebraictwisted}
  Define the \emph{algebraic twisted Roe algebra} $\mathbb{C}[\Gamma, \mathcal{A}(\mathcal{H})\hat{\otimes}B]^{\Gamma} $ as the set of all the functions $T:\Gamma \times \Gamma \to \mathcal{A}(\mathcal{H})\hat{\otimes}B\hat{\otimes}\mathcal{K}$ such that for any  $T_{g,h} \in \mathcal{A}(\mathcal{H})\hat{\otimes}B\hat{\otimes}\mathcal{K}$:
  \begin{enumerate}[label=(\arabic*)]
    \item $\exists\ M>0$ such that $\|T_{g,h}\| \le M$;
    \item $\exists\ L>0$ such that $|\{ g \mid T_{g,h}\neq 0 \}|, | \{ h \mid T_{g,h} \neq 0 \}|<L$;
    \item $\exists\ r_{1}\ge 0$ such that $T_{g,h}=0$ if $d(g,h)>r_{1}$;
    \item $\exists\ r_{2}\ge 0$ such that $\operatorname{supp}(T_{g,h}) \subset B(j(g),r_{2})$;
    \item $T$ is $\Gamma$-equivariant, i.e., $T_{g,h}=\gamma\cdot (T_{\gamma^{-1}g,\gamma^{-1}h})$ for any $\gamma \in \Gamma$.
  \end{enumerate}
\end{definition}

\begin{remark}
  We should point out that the condition (4) in Definition \ref{defalgebraictwisted} is implicitly contained in the definition of standard equivariant Roe algebra. Indeed, since $\Gamma$ has a proper length function, for any radius $r_{1}$, the set $B_{\Gamma}(e,r_{1})$ is finite, where $e\in \Gamma$ denotes the identity element of $\Gamma$. Consider an operator $T$ with propagation less than $r_{1}$, if $T_{e,h}\neq 0$, then $h \in B_{\Gamma}(e,r_{1})$. If there exists some $r_{2}>0$ such that $\operatorname{supp}(T_{e,h})\subset B(j(e),r_{2})$, by the $\Gamma$-equivariance and also by the action being isometric, we have that $\operatorname{supp}(T_{g,h}) \subset B(j(g),r_{2})$ for all $g,h \in \Gamma$. Since $C_{0}(\mathbb{R}_{+} \times \mathcal{H})$ is dense in $\mathcal{A}(\mathcal{H})$, we can always assume the extension of such $r_{2}$ up to a small error.
\end{remark}

Parrallel to the equivariant Roe algebra as in Definition \ref{usual-prop-spt}, we can define the support of an operator $T \in \mathbb{C}[\Gamma,\mathcal{A}(\mathcal{H})\hat{\otimes}B]^{\Gamma} $ by 
$$
\operatorname{supp}(T)=\{ (g,h,v)\in \Gamma \times \Gamma \times (\mathbb{R}_{+}\times \mathcal{H}) \mid v \in \operatorname{supp}(T_{g,h}) \}.
$$

The algebraic twisted Roe algebra $\mathbb{C}[\Gamma, \mathcal{A}(\mathcal{H})\hat{\otimes}B]^{\Gamma} $ admits a natural *-representation on the Hilbert module $E_{\mathcal{A}}$, given by 
$$
(T\xi)(g)=\sum_{h\in \Gamma}T_{g,h}\xi(h).
$$ 

Then the formal definition is as follows:
\begin{definition}\label{deftwisted}
  The \emph{reduced twisted Roe algebra} $C^{*}(\Gamma, \mathcal{A}(\mathcal{H})\hat{\otimes}B)^{\Gamma} $ is the completion of the algebraic twisted Roe algebra under the operator norm in $\mathcal{L}(E_{\mathcal{A}})$, the $C^*$-algebra of all adjointable $\mathcal{A}(\mathcal{H})\hat{\otimes}B\hat{\otimes}\mathcal{K}$-homomorphism.
\end{definition}

By Lemma \ref{normcontrol}, we can also define the universal norm by taking the supremum over all $*$-representations of the algebraic twisted Roe algebra into bounded operators on Hilbert spaces.

\begin{definition}
  The \emph{maximal twisted Roe algebra} $C^{*}_{\max}(\Gamma, \mathcal{A}(\mathcal{H})\hat{\otimes}B)^{\Gamma} $
  is given by the completion of $\mathbb{C}[\Gamma, \mathcal{A}(\mathcal{H})\hat{\otimes}B]^{\Gamma} $ under the norm  
  $$
  \|T\|_{\max}:=\sup_{\phi,\mathcal{H}_{\phi}} \|\phi(T)\|_{\mathcal{B}(\mathcal{H}_{\phi})},
  $$ where $(\phi,\mathcal{H}_{\phi})$ runs through all $*$-representations of $\mathbb{C}[\Gamma, \mathcal{A}(\mathcal{H})\hat{\otimes}B]^{\Gamma} $ on Hilbert spaces.
\end{definition}

By the universal property of the maximal completion, we have the canonical quotient map: 
$$
\lambda: C^{*}_{\max} (\Gamma, \mathcal{A}(\mathcal{H})\hat{\otimes}B)^{\Gamma}  \to C^{*}(\Gamma, \mathcal{A}(\mathcal{H})\hat{\otimes}B)^{\Gamma} .
$$ 
This map induces a homomorphism on $K$-theory:
$$
\lambda_{*}: K^{}_{*}(C^{*}_{\max} (\Gamma, \mathcal{A}(\mathcal{H})\hat{\otimes}B)^{\Gamma} ) \to K^{}_{*}(C^{*}(\Gamma, \mathcal{A}(\mathcal{H})\hat{\otimes}B)^{\Gamma} ).
$$


\section{The Proof of the Main Theorem}\label{section4}

In this section, we prove Theorem \ref{A}. The argument follows the Dirac-dual-Dirac method, by employing the geometric construction developed by Yu in \cite{Yu-2000} for spaces that admit coarse embeddings into Hilbert space. We shall first construct the following diagram: 
\begin{equation}\label{cmmu}
    \begin{tikzcd}
	{   K_{*}(C^{*}_{\max}(\Gamma, B)^{\Gamma} \hat{\otimes}\mathcal{S})} && {   K_{*}(C^{*}(\Gamma, B)^{\Gamma} \hat{\otimes}\mathcal{S})} \\
	{ K_{*}(C^{*}_{\max}(\Gamma,\mathcal{A}(\mathcal{H})\hat{\otimes}B)^{\Gamma} )} && { K_{*}(C^{*}(\Gamma,\mathcal{A}(\mathcal{H})\hat{\otimes}B)^{\Gamma} )} \\
	{   K_{*}(C^{*}_{\max}(\Gamma, B)^{\Gamma} \hat{\otimes}\mathcal{S})} && {   K_{*}(C^{*}(\Gamma, B)^{\Gamma} \hat{\otimes}\mathcal{S})}
	\arrow["{{\lambda_{*}}}", from=1-1, to=1-3]
	\arrow["{{\beta_{*}}}"', from=1-1, to=2-1]
	\arrow["{{\beta_{*}}}", from=1-3, to=2-3]
	\arrow["{{\lambda_{*}}}", from=2-1, to=2-3]
	\arrow["{\alpha_{*}}"', from=2-1, to=3-1]
	\arrow["{\alpha_{*}}", from=2-3, to=3-3]
	\arrow["{\lambda_{*}}", from=3-1, to=3-3]
\end{tikzcd}
\end{equation}
The proof consists of two parts: The vertical direction follows Yu's construction in \cite{Yu-2000}. The horizontal isomorphism in the middle row comes from the cutting-and-pasting argument that reduces the problem to subgroups commensurable with $\Lambda$.


\subsection{The  Bott map and the Dirac map}\label{section4.1}

We begin with the construction of the Bott map and the Dirac map. We shall first recall the Dirac and dual-Dirac element introduced by Higson, Kasparov and Trout in \cite{higson-kasparov-trout-1998}. 

For each finite-dimensional affine subspace $V\subset \mathcal{H}$, the Clifford algebra $\mathrm{Cl}(V^0)$ is a finite-dimensional $C^*$-algebra with a canonical basis, where $V^0$ is the linear part of $V$. The basis gives a Hilbert space structure of $\mathrm{Cl}(V^0)$, denoted by $\mathfrak{h}_{V}$. Let $\mathcal{L}(V):=L^{2}(V,\mathfrak{h}_V)$ and let $\mathfrak{s}(V)$ be the Schwarz subspace of $\mathcal{L}(V)$.
If $W\subseteq \mathcal H$ is a finite-dimensional affine subspace containing $V$. There is a canonical inclusion
$$\mathcal{L}(V)\to\mathcal{L}(W)$$
by $\xi\mapsto \xi\otimes G$, where $G\in \mathcal{L}(V')$ is the Gauss function on $V'$ and $V'\oplus V=W$. Passing to the direct limit, we have $\mathcal{L}(\mathcal{H})=\varinjlim\mathcal{L}(V) $ and its dense subspace $\mathfrak{s}(\mathcal{H}):=\varinjlim \mathfrak{s}(V) $, where $V$ runs over all finite-dimensional subspaces of $\mathcal H$.

For any vector $v_0 \in V$, define $C_{V,v_0}$ to be the unbounded function with center $v_0$
$$C_{V,v_0}(v)=v-v_0\in V\subseteq \mathrm{Cl}(V).$$
It acts on $\mathcal{L}(V)$ as an unbounded operator by
$$
(C_{V,v_0}\cdot \xi) (v):=(v-v_{0})\cdot \xi (v), \quad \text{for all }  \xi\in \mathfrak{s}(V),
$$
By taking the direct limit, we extend to a multiplier $C$ on $\mathcal{A}(\mathcal{H})$.

Define the Bott map between the algebraic Roe algebra and the algebraic twisted Roe algebra by
$$
\begin{aligned}
 \beta_{t}: \mathbb{C}[\Gamma,B]^{\Gamma} \hat{\otimes} \mathcal{S} &\to \mathbb{C}[\Gamma,\mathcal{A}(\mathcal{H})\hat{\otimes}B]^{\Gamma} ,\\
  \beta_{t}(T \hat{\otimes} f)_{g,h}&=T_{g,h} \hat{\otimes} C_{g,h}(f_{t}),
\end{aligned}
$$
where $g,h\in \Gamma$, $C_{g,h}(f_{t})=f_{t}(X\hat{\otimes}1+1\hat{\otimes}C_{j(g)})$, $C_{j(g)}$ denotes the Clifford multiplier associated with $j(g)$, and $f_t(r)=f(r/t)$. It is direct to check that $\gamma\cdot f_{t}(X\hat{\otimes}1+1\hat{\otimes}C_{j(g)})=f_{t}(X\hat{\otimes}1+1\hat{\otimes}C_{j(\gamma g)})$, which implies that the Bott map $\beta_{t}$ is $\Gamma$-equivariant.

\begin{lemma}\label{Bott-asymptotic}
  The Bott map $\beta_{t}$ extends to asymptotic morphisms:
  $$
  \begin{aligned}
   \beta: C^{*}(\Gamma,B)^{\Gamma} \hat{\otimes} \mathcal{S} &\to C^{*}(\Gamma,\mathcal{A}(\mathcal{H})\hat{\otimes}B)^{\Gamma} ,\\
   \beta_{\max}: C^{*}_{\max}(\Gamma,B)^{\Gamma} \hat{\otimes} \mathcal{S} &\to C^{*}_{\max}(\Gamma,\mathcal{A}(\mathcal{H})\hat{\otimes}B)^{\Gamma} .
  \end{aligned}
  $$
\end{lemma}
\begin{proof}
  To prove that $\beta$ and $\beta_{\max}$ are asymptotic morphisms, we need to verify that the family of maps $\beta_{t}$ is uniformly bounded and asymptotically multiplicative.
  
  For the uniform boundedness, note that for any fixed $t$, the operator $f_{t}(X\hat{\otimes}1+1\hat{\otimes}C_{j(g)})$ is bounded in $\mathcal{A}(\mathcal{H})$. Since $T$ has finite propagation and coefficients in $B$, the norm of $\beta_{t}(T\hat{\otimes}f)$ is controlled by the norms of $T$ and $f$.
  
  To see $\beta_{t}$ is asymptotic multiplicative, let $T, T' \in \mathbb{C}[\Gamma, B]^{\Gamma} $ and $f, f' \in \mathcal{S}$. Consider the product:
  \begin{equation}\label{eq:beta_prod}
  (\beta_{t}(T\hat{\otimes}f)\beta_{t}(T'\hat{\otimes}f'))_{g,h} = \sum_{k \in \Gamma} T_{g,k}T'_{k,h} \hat{\otimes} C_{g,k}(f_{t})C_{k,h}(f'_{t}).
  \end{equation}
  Recall that $C_{g,k}(f_{t}) = f_{t}(\C_{g})$ where $\C_{g} = X\hat{\otimes}1+1\hat{\otimes}C_{j(g)}$. Thus,
  $$
  C_{g,k}(f_{t})C_{k,h}(f'_{t}) = f_{t}(\C_{g})f'_{t}(\C_{k}).
  $$
  On the other hand,
  \begin{equation}\label{eq:prod_beta}
  \beta_{t}((T\hat{\otimes}f)(T'\hat{\otimes}f'))_{g,h} = \sum_{k \in \Gamma} T_{g,k}T'_{k,h} \hat{\otimes} (f\cdot f')_{t}(\C_{g}).
  \end{equation}
  The difference between \eqref{eq:beta_prod} and \eqref{eq:prod_beta} involves the term $f_{t}(\C_{g})f'_{t}(\C_{k}) - (f\cdot f')_{t}(\C_{g})$. Since $T$ has finite propagation, the distance $d(g,k)$ is uniformly bounded for all $k$ such that $T_{g,k} \neq 0$. By coarse embedding, $\|j(g) - j(k)\|\leq \rho_+(d(g\Lambda,k\Lambda))\leq \rho_+(d(g,k))$. Consequently, the difference of the operators $\C_{g} - \C_{k} = 1\hat{\otimes}C_{j(g)-j(k)}$ is uniformly bounded for all pairs $(g,k)$ such that $T_{g,k}\ne 0$.

Consider the generators $t\mapsto\frac{1}{t\pm i}\in C_0(\R)$. Let $R_{t}(\C_g):=(t^{-1}\C_k+i)^{-1}$. The difference $t^{-1}\C_{g} - t^{-1}\C_{k}$ converges to 0 in norm as $t \to \infty$. Following a direct calculation.
$$
R_{t}(\C_{g}) - R_{t}(\C_{k}) = R_{t}(\C_{g}) (t^{-1}(\C_{k}-\C_{g})) R_{t}(D_{k}),
$$
which converges to 0 in norm. Similar argument also holds for $t\mapsto \frac{1}{t- i}$. Thus, for any $f\in C_0(\R)$, we conclude that $\| f_{t}(\C_{g}) - f_{t}(\C_{k}) \| \to 0$. Thus $\|f_{t}(\C_{g})f'_{t}(\C_{k}) -f_{t}(\C_{g})f'_{t}(\C_{g})\|\to 0$ as $t\to \infty$.
  
Thus, $\beta_{t}$ defines an asymptotic morphism on the algebraic tensor product. By continuity, it extends to the maximal and also to the reduced completions.
\end{proof}

Following from the above lemma, we obtain the Bott map for both maximal and reduced Roe algebras on $K$-theory: 
$$
 \beta_{*}: K_{*}(C^{*}_{(\max)}(\Gamma, B)^{\Gamma} \hat{\otimes}\mathcal{S}) \to K_{*}(C^{*}_{(\max)}(\Gamma,\mathcal{A}(\mathcal{H})\hat{\otimes}B)^{\Gamma} ).
$$


Next, we recall the construction of the Dirac element $\alpha_{*}$. For more details, the reader can refer to \cite{Yu-2000} and \cite{fu-wang-2016}. 
On a finite-dimensional affine subspace $V$, the Dirac operator $D_{V}$ on $\mathcal{L}(V)$ with domain $\mathfrak{s}(V)$ is defined by:
$$
D_{V}(\xi)=\sum_{i=1}^n(-1)^{deg(\xi)}\frac{\partial(\xi)}{\partial x_i}v_i, \quad \forall \xi \in \mathfrak{s}(V),
$$
where $\{ v_{1}, v_{2}, \ldots , v_{n} \}$ is an orthonormal basis of $V^{0}$, and $x_{i}$ denotes the coordinate function on $V$ corresponding to the basis vector $v_{i}$. 

Fix a base point $x\in \Gamma$. For each $n\in \mathbb{N}$, define the finite-dimensional affine subspace
$$
\begin{aligned}
 W_{n}(x)&=j(x)+\operatorname{span} \{ j(g)-j(x) \mid g\in \Gamma, |g^{-1}x| \le n \}\\
 &= h(x\Lambda)+ \operatorname{span} \{ h(g\Lambda)-h(x\Lambda) \mid g \in \Gamma, |g^{-1}x| \le n \},
\end{aligned}
$$
where $h$ is the coarse embedding proved by Proposition \ref{Keyprop} and $j(\gamma)=h(\gamma \Lambda)$ is the lifted map introduced in Section \ref{section3.1}. Because the length function $|\cdot |$ on $\Gamma$ is proper, each $W_{n}(x)$ is indeed finite-dimensional. Set $W(x)=\bigcup_{k=1}^{\infty}W_{k}(x)$. Since $\mathcal{H}$ the closure of the linear span of $j(\Gamma)$, the subspace $W(x)$ is dense in $\mathcal{H}$. 

Since $\Gamma$ acts on $\mathcal{H}$ by affine isometries, for any $\gamma \in \Gamma$, we have $W_{n}(\gamma x)=\gamma\cdot W_{n}(x)$. Consequently, $\{ W_{n}(x) \}$ forms a directed system of finite-dimensional affine subspaces. Set $W_{0}(x)=\{j(x)\}$. For each $n\in \mathbb{N}$, let $V_{n}(x)$ be the orthogonal complement of $W_{n}(x)$ in $W_{n+1}(x)$, which is a finite-dimensional linear subspace of $\mathcal{H}$. Let $V(x)=\bigoplus _{n=1}^{\infty}V_{n}(x)$, $V(x)$ is a dense linear subspace of $\mathcal{H}$. We have the following decompositions
$$
\mathcal{A}(W_{n+1}(x)) \cong \mathcal{A}(W_{n}(x))\hat{\otimes}\mathcal{A}(V_{n}(x)),\quad \mathcal{L}(W_{n+1}(x)) \cong \mathcal{L}(W_{n}(x))\hat{\otimes}\mathcal{L}(V_{n}(x)).
$$

Write $D_{V_{n}(x)}$ simply as $D_{n}$. Let $C_{n}=C_{V_n(x),0}$ be the Clifford multiplier on $V_n$. For each $n\in \mathbb{N}$ and $t\ge 1$, define an unbounded operator $B_{n,t}(x)$ on $\mathcal{L}(V(x))$ by 
$$
\begin{aligned}
  B_{n,t}(x)&=t_{0}D_{0}+t_{1}D_{1}+\cdots+t_{n-1}D_{n-1}+t_{n}(D_{n}+C_{n}) \\ 
  &+t_{n+1}(D_{n+1}+C_{n+1})+\cdots, 
\end{aligned}
$$
where $t_{i}=1+t^{-1}i$. According to the discussion in Section 7 of \cite{Yu-2000}, each $B_{n,t}(x)$ is essentially self-adjoint and has compact resolvent. 

The group $\Gamma$ acts on $\mathcal{L}(\mathcal{H})$ by:
$$
(\gamma\cdot \xi)(v)=\widetilde{\gamma}(\xi(\gamma^{-1}v)), \quad \text{for all } \gamma \in \Gamma, v \in \mathcal{H}, \xi \in \mathcal{L}(\mathcal{H}),
$$
where $\widetilde{\gamma}$ is the unitary operator on the spinor fiber induced by the linear part of the affine isometry $\gamma$. This action is compatible with the action on the algebra $\mathcal{A}(\mathcal{H})$ defined in Section \ref{section3.2}. Because the construction of the subspaces $V_n(x)$ is canonical and $\Gamma$ acts by affine isometries, we have $W_k(\gamma x) = \gamma W_k(x)$, and consequently each $B_{n,t}(x)$ satisfies the equivariant property:
$$
B_{n,t}(\gamma x) = \gamma B_{n,t}(x) \gamma^{-1}.
$$

Let $\mathcal{K}_{\mathcal{L}}$ denote the algebra of compact operators in $\mathcal{L}(\mathcal{H})$. Define $\mathbb{C}[\Gamma,\mathcal{K}_{\mathcal{L}}\hat{\otimes}B]^{\Gamma} $ as the $C^{*}$-algebra of functions $T:\Gamma \times \Gamma \to \mathcal{K}_{\mathcal{L}}\hat{\otimes}B\hat{\otimes}\mathcal{K}$ such that for any $T_{g,h}\in \mathcal{K}_{\mathcal{L}}\hat{\otimes}B\hat{\otimes}\mathcal{K}$:
\begin{enumerate}[label=(\arabic*)]
  \item $\exists \ M > 0$ such that $\|T_{g,h}\| \le  M$;
  \item $\exists \ L>0$ such that $|\{ g \mid T_{g,h} \neq 0 \}|$, $| \{ h \mid T_{g,h} \neq 0 \}| < L$ ;
  \item $\exists \ R>0 $ such that $T_{g,h}=0$ if $d(g,h)>R$;
  \item $T$ is $\Gamma$-invariant, i.e., $T_{\gamma g, \gamma h}=\gamma \cdot T_{g,h}$ for all $\gamma \in \Gamma$.
\end{enumerate}
The algebra $\mathbb{C}[\Gamma,\mathcal{K}_{\mathcal{L}}\hat{\otimes}B]^{\Gamma} $ admits a canonical *-representation on the Hilbert module $\ell^{2}(\Gamma)\hat{\otimes}B\hat{\otimes}\mathcal{K}$. Denote by $C^{*}(\Gamma,\mathcal{K}_{\mathcal{L}}\hat{\otimes}B)^{\Gamma} $ and $C^{*}_{\max}(\Gamma,\mathcal{K}_{\mathcal{L}}\hat{\otimes}B)^{\Gamma} $ the completions of $\mathbb{C}[\Gamma,\mathcal{K}_{\mathcal{L}}\hat{\otimes}B]^{\Gamma} $ with respect to the reduced and maximal norms, respectively.

For every $n\in \mathbb{N}_{+}$, $t\ge 1$ and $x \in \Gamma$, we define
$$
\begin{aligned}
\theta^{n}_{t}(x):  & \mathcal{A}(W_{n}(x))\hat{\otimes}B\hat{\otimes}\mathcal{K} \longrightarrow \mathcal{K}_{\mathcal{L}}\hat{\otimes}B\hat{\otimes}\mathcal{K} \\
& \theta^{n}_{t}(x)(f \hat{\otimes} a \hat{\otimes} k)=f_{t}\bigl(B_{n,t}(x)\bigr)M_{a_{t}}\hat{\otimes}k,
\end{aligned}
$$
where $f\in \mathcal{S}$, $a\in \mathcal{C}(W_{n}(x))$, $k\in B \hat{\otimes}\mathcal{K}$, and $a_{t}(v)=a(t^{-1}v)$. The multiplier $M_{a_{t}}$ is defined by 
$$
(M_{a_{t}}\xi)(v)=a_{t}(w)\cdot \xi(v), \quad \forall v \in \mathcal{H}, \xi \in \mathcal{L}(\mathcal{H}),
$$
where $w$ denotes the orthogonal projection of $v$ onto the affine subspace $W_{n}(x)$. 

According to Lemma 5.8 in \cite{higson-kasparov-trout-1998}, we know that $\theta^{n}_{t}(x)(f \hat{\otimes} a \hat{\otimes} k)\in \mathcal{K}_{\mathcal{L}}\hat{\otimes}B\hat{\otimes}\mathcal{K}$. And for any $T \in \mathbb{C}[\Gamma,\mathcal{A}(\mathcal{H})\hat{\otimes}B]^{\Gamma} $, choose a non-negative integer $n_{0}$ such that for all $g,h \in \Gamma$, there exists $T'_{g,h}\in \mathcal{A}(W_{n_{0}}(g))\hat{\otimes}B\hat{\otimes}\mathcal{K}$ satisfying
$$
T_{g,h}=\beta_{n_{0},g}(T'_{g,h}),
$$
where $\beta_{n_{0},g}$ extends the inclusion $\mathcal{A}(W_{n_{0}}(g))\to \mathcal{A}(\mathcal{H})$ defined in \eqref{extend}. We then define the Dirac operator on the algebraic twisted Roe algebra by
$$
(\alpha_{t}(T))_{g,h}=\theta_{t}^{n_{0}}(g)(T'_{g,h}),\quad \text{for all } T\in \mathbb{C}[\Gamma,\mathcal{A}(\mathcal{H})\hat{\otimes}B]^{\Gamma} , g,h \in \Gamma.
$$
Proposition 4.2 in \cite{higson-kasparov-trout-1998} shows that for a different choice of $n_{0}$, the elements $\alpha_{t}(T)$ converge to the same limit in norm as $t \to \infty$.

\begin{lemma}
The element $\alpha_{t}(T)$ is $\Gamma$-invariant, i.e., $\alpha_{t}(T) \in \mathbb{C}[\Gamma,\mathcal{K}_{\mathcal{L}}\hat{\otimes}B]^{\Gamma} $.
\end{lemma}
\begin{proof}
We need to verify that 
$$
(\alpha_{t}(T))_{\gamma g, \gamma h}=\gamma \cdot (\alpha_{t}(T))_{g,h}, \quad \text{for all} \gamma, g,h \in \Gamma.
$$
First, recall that $T$ is $\Gamma$-invariant, and $T'_{\gamma g, \gamma h}=\gamma \cdot T'_{g,h}$ under the natural isomorphism induced by the affine isometry $\gamma: W_{n_{0}}(g) \to W_{n_{0}}(\gamma g)$.

Next, consider $\theta_{t}^{n_{0}}(x)$ constructed from the operator $B_{n_{0},t}(x)$ and the multiplier $M_{a_{t}}$. We have already shown that $B_{n_{0},t}(\gamma g) = \gamma B_{n_{0},t}(g) \gamma^{-1}$. By the covariance of the Clifford multiplication, the multiplier $M_{a_{t}}$ satisfies $M_{(\gamma \cdot a)_{t}} = \gamma M_{a_{t}} \gamma^{-1}$.

Combining these observations, we obtain
$$
\begin{aligned}
(\alpha_{t}(T))_{\gamma g, \gamma h} &= \theta_{t}^{n_{0}}(\gamma g)(T'_{\gamma g, \gamma h}) \\
&= \theta_{t}^{n_{0}}(\gamma g)(\gamma \cdot T'_{g,h}) \\
&= \gamma \cdot \theta_{t}^{n_{0}}(g)(T'_{g,h}) \\
&= \gamma \cdot (\alpha_{t}(T))_{g,h}.
\end{aligned}
$$
Thus, $\alpha _{t}(T)$ is $\Gamma$-equivariant, which completes the proof.
\end{proof}

\begin{lemma}\label{dirac}
  The Dirac map $\alpha_{t}$ extends to asymptotic morphisms
  $$
  \begin{aligned}
   \alpha: C^{*}(\Gamma,\mathcal{A}(\mathcal{H})\hat{\otimes}B)^{\Gamma}  &\to C^{*}(\Gamma,\mathcal{K}_{\mathcal{L}}\hat{\otimes}B)^{\Gamma} , \\
   \alpha_{\max}: C^{*}_{\max}(\Gamma,\mathcal{A}(\mathcal{H})\hat{\otimes}B)^{\Gamma}  &\to C^{*}_{\max}(\Gamma,\mathcal{K}_{\mathcal{L}}\hat{\otimes}B)^{\Gamma} .
  \end{aligned}
  $$
\end{lemma}

\begin{proof}
We show that $\alpha$ and $\alpha_{\max}$ are asymptotic morphisms. This requires checking that $\{\alpha_t\}$ is uniformly bounded and asymptotically multiplicative.

The uniform boundedness follows from the fact that for a fixed $t$, $\theta_t^n(g)$ is a contractive map (composition of functional calculus and multiplication) and $T$ has finite propagation.

To show asymptotic multiplicativity, take $T, S \in \mathbb{C}[\Gamma, \mathcal{A}(\mathcal{H})\hat{\otimes}B]^{\Gamma} $. Choose $n$ large enough so that the coefficients of $T$, $S$ and $TS$ are all supported in $W_n$ (modulo $\Gamma$). We must prove
$$
\lim_{t \to \infty} \bigl\| \alpha_{t}(T)\alpha_{t}(S) - \alpha_{t}(TS) \bigr\| = 0.
$$

Let $R>0$ be such that $T_{g,w}=0$ and $S_{w,h}=0$ whenever $d(g,w)>R$ or $d(w,h)>R$. The $(g,h)$-entry of the difference is
$$
(\alpha_{t}(TS))_{g,h} - (\alpha_{t}(T)\alpha_{t}(S))_{g,h} 
= \sum_{w} \Bigl( \theta_t^n(g)(T'_{g,w} S'_{w,h}) - \theta_t^n(g)(T'_{g,w}) \theta_t^n(w)(S'_{w,h}) \Bigr).
$$

Since the sum is finite, it suffices to show that each term tends to $0$ uniformly. Write
$$
\Delta_{g,w,h}=A_{g,w,h}+B_{g,w,h},
$$
where
$$
\begin{aligned}
A_{g,w,h}&=\theta_t^n(g)(T'_{g,w} S'_{w,h})-\theta_t^n(g)(T'_{g,w})\theta_t^n(g)(S'_{w,h}),\\ 
B_{g,w,h}&=\theta_t^n(g)(T'_{g,w})\bigl(\theta_t^n(g)(S'_{w,h})-\theta_t^n(w)(S'_{w,h})\bigr).
\end{aligned}
$$

We shall estimate $A_{g,w,h}$ and $B_{g,w,h}$ separately.

\emph{Estimation of $A_{g,w,h}$.} Write $T'_{g,w}=f_1\hat{\otimes}a_1\hat{\otimes}k_1$ and $S'_{w,h}=f_2\hat{\otimes}a_2\hat{\otimes}k_2$. The term $A_{g,w,h}$ essentially involves the commutator $[M_{(a_1)_t}, (f_2)_t(B_{n,t})]$. Because $[M_{(a_1)_t}, B_{n,t}]$ equals the Clifford multiplication by $c(d(a_1)_t)$ and $d(a_1)_t = t^{-1}(da_1)_t$, we have $\|[M_{(a_1)_t}, B_{n,t}]\| = O(t^{-1})$. Using the Fourier representation
$$
f_2(t^{-1}B_{n,t}) = \frac{1}{2\pi} \int \hat{f_2}(\xi) e^{i\xi t^{-1}B_{n,t}}\,d\xi,
$$
one finds
$$
\bigl\| [M_{(a_1)_t}, f_2(t^{-1}B_{n,t})] \bigr\| \le C t^{-2}.
$$
Hence $\|A_{g,w,h}\|\to 0$ as $t\to\infty$.

\emph{Estimation of $B_{g,w,h}$.} We estimate $\| \theta_t^n(g)(u) - \theta_t^n(w)(u) \|$ for $u = S'_{w,h}$ with $d(g,w) \le R$. The operator $B_{n,t}(g)$ contains a Clifford part $C_{n,t}(g)=\sum_{k=n}^\infty t_k C_k$, where $C_k$ acts on $V_k(g)$ by Clifford multiplication of the vector $v_k$. Up to lower order terms, $C_{n,t}(g)$ approximates Clifford multiplication by $v-P_{W_n(g)}(v)$, the projection of $v-j(g)$ onto the orthogonal complement of $W_n(g)$.

When we replace $g$ by $w$, the difference in the Clifford parts essentially reduces to Clifford multiplication by the constant vector $j(w)-j(g)$:
$$
C_{n,t}(g)-C_{n,t}(w)\approx c\bigl(j(w)-j(g)\bigr).
$$
Because $d(g,w)\le R$, the vector $j(w)-j(g)$ has bounded norm. After scaling by $t^{-1}$ we obtain
$$
t^{-1}\bigl(B_{n,t}(g)-B_{n,t}(w)\bigr)=O(t^{-1}),
$$
which tends to $0$ in norm as $t\to\infty$. Since the functional calculus $D\mapsto f(D)$ is uniformly continuous on bounded sets of self-adjoint operators,
$$
\bigl\| f\bigl(t^{-1}B_{n,t}(g)\bigr)-f\bigl(t^{-1}B_{n,t}(w)\bigr) \bigr\|\longrightarrow 0.
$$
Hence $\|B_{g,w,h}\|\to 0$ as $t\to\infty$.

Combining the estimations for $A_{g,w,h}$ and $B_{g,w,h}$, we obtain the required asymptotic multiplicativity.
\end{proof}

Parallel to the Bott map $\beta_{*}$, we can define the Dirac map 
$$
  \alpha_{*}: K^{*}(C^{*}_{(\max)}(\Gamma,\mathcal{A}(\mathcal{H})\hat{\otimes}B)^{\Gamma} ) \to K_{*}( C^{*}_{(\max)}(\Gamma,\mathcal{S}\hat{\otimes}B)^{\Gamma} ).
$$

The following proposition is then immediate.
\begin{proposition}\label{2st thm}
  The composition $\alpha_{*} \circ \beta_{*}$ is the identity map on $K$-theory.
\end{proposition}
\begin{proof}
  This is a direct consequence of Proposition 7.7 in \cite{Yu-2000}.
\end{proof}

\subsection{The cutting-and-pasting argument}

In this subsection, we prove the middle horizontal map in \eqref{cmmu} is an isomorphism. Before proceeding with the detailed proof, we outline our general strategy.

Since $\mathcal{A}(\mathcal{H})$ is a proper algebra, we can employ a ``cutting-and-pasting'' technique similar to that used in \cite[Theorem 13.1]{Guentner-Higson-Trout-2000} to reduce the global problem to the stabilizers of the action. In the setting of \cite{Guentner-Higson-Trout-2000}, the group $\Gamma$ is assumed to act properly on the Hilbert space $\mathcal{H}$. Consequently, all stabilizers are finite groups. Since the $K$-amenability (and the Baum-Connes conjecture) for finite groups is well understood, the proof can be completed directly in that case.

In our context, however, we only assume that the quotient space $\Gamma/\Lambda$ admits a $\Gamma$-equivariant coarse embedding into $\mathcal{H}$. As a result, the stabilizers for the affine action on $\mathcal{H}$ are subgroups of $\Gamma$ commensurable with $\Lambda$. When $\Lambda$ is infinite, these stabilizers are typically infinite, which introduces new challenges. Note that while the stabilizers are quasi $K$-amenable, they are generally not amenable. This leads to a technical issue regarding the maximal norm. Since the maximal norm is not hereditary, the restriction of the global maximal norm to a local subalgebra does not necessarily coincide with the intrinsic maximal norm of that subalgebra. Therefore, we must handle the completions carefully in the following reduction.

Specifically, for any $r>0$, let $O_{r}(g)$ denote the open ball centered at $(j(g), 0)$ in the product space $\mathbb{R}_+  \times \mathcal{H}$:
\[
O_{r}(g) := B((j(g),0), r) = \left\{ (t,v) \in \mathbb{R}_+  \times \mathcal{H} \mid \|v-j(g)\|^2 + t^2 < r^2 \right\}.
\]
Since the action is $\Gamma$-equivariant and isometric on the affine space, we have $O_{r}(g) = g \cdot O_{r}(e)$. We define the $\Gamma$-invariant open set $O_{r}$ as the union of these orbits:
\[
O_{r} := \Gamma \cdot O_{r}(e) = \bigcup_{g \in \Gamma} O_{r}(g).
\]
Let $\mathcal{A}(\mathcal{H})_{O_{r}}$ denote the subalgebra of $\mathcal{A}(\mathcal{H})$ consisting of functions supported in the subset $O_{r}$. Correspondingly, let $\mathbb{C}[\Gamma, \mathcal{A}(\mathcal{H})_{O_r} \otimes B]^\Gamma$ be the subalgebra of $\mathbb{C}[\Gamma, \mathcal{A}(\mathcal{H}) \otimes B]^\Gamma$ consisting of operators supported in $\Gamma \times \Gamma \times O_{r}$. It follows directly from the definition that:
\[
\mathbb{C}[\Gamma, \mathcal{A}(\mathcal{H}) \otimes B]^\Gamma = \lim_{r \to \infty} \mathbb{C}[\Gamma, \mathcal{A}(\mathcal{H})_{O_r} \otimes B]^\Gamma,
\]
where the limit is taken over the directed family of open sets $\{O_r\}_{r>0}$.

However, a subtlety arises when considering the maximal completions. For each $r$, let
\[
\phi_{\max, r}: \mathbb{C}[\Gamma, \mathcal{A}(\mathcal{H})_{O_r} \otimes B]^\Gamma \to \mathcal{B}(\mathcal{H}_{\max, r})
\]
denote the universal representation of the local algebra. For $r' > r > 0$, the natural inclusion implies that the algebra at radius $r$ is a subalgebra of that at radius $r'$. However, the restriction of the universal representation
\[
\phi_{\max, r'} \big|_{r}: \mathbb{C}[\Gamma, \mathcal{A}(\mathcal{H})_{O_r} \otimes B]^\Gamma \to \mathcal{B}(\mathcal{H}_{\max, r'})
\]
may not coincide with the intrinsic maximal representation $\phi_{\max, r}$. Consequently, even assuming the $K$-amenability of the stabilizers, the cutting-and-pasting process does not automatically ensure that the diagram between maximal and reduced completions commutes or remains isometric at each step. To address these technical challenges, we need the following technical lemma.

\begin{lemma}\label{lem: sandwich}
Let $\Lambda$ be an a-T-menable discrete group and let $B$ be a $\Lambda$-$C^*$-algebra. Suppose that $\phi: \mathbb{C}[\Lambda, B]^\Lambda \to \mathcal{B}(\mathcal{H}_\phi)$ is a $*$-representation such that the associated $C^*$-completion, denoted by $C^*_\phi(\Lambda, B)^\Lambda$, lies between the maximal and reduced crossed products, i.e., the canonical quotient map $\lambda: C^*_{\max}(\Lambda, B)^\Lambda \to C^*_\phi(\Lambda, B)^\Lambda$ factors through $C^*_\phi(\Lambda, B)^\Lambda$:
\[
\lambda: C^*_{\max}(\Lambda, B)^\Lambda  \xrightarrow{\pi_\phi} C^*_\phi(\Lambda, B)^\Lambda \xrightarrow{\lambda_\phi} C^*(\Lambda, B)^\Lambda,
\]
such that $\lambda = \lambda_\phi \circ \pi_\phi$. Then the induced map  $(\pi_\phi)_*: K_*(C^*_{\max}(\Lambda, B)^\Lambda)\to K_*(C^*_{\phi}(\Lambda, B)^\Lambda)$ and $(\lambda_\phi)_*: K_*(C^*_{\phi}(\Lambda, B)^\Lambda)\to K_*(C^*(\Lambda, B)^\Lambda)$ are both isomorphisms.
\end{lemma}

\begin{proof}
  Since $\Lambda$ is a-T-menable, there exists an infinite-dimensional Hilbert space $\mathcal{H}_{\Lambda}$ admitting a proper affine isometric action of $\Lambda$. Parallel to the construction in Section \ref{section4.1}, we can define the Bott map and Dirac map for both maximal and reduced crossed products with coefficients in $B$:
  $$
  \begin{aligned}
   \beta_{*}:K_{*}(C^{*}_{\max}(\Lambda, B)^{\Lambda} \hat{\otimes}\mathcal{S}) \to K_{*}(C^{*}_{\max}(\Lambda,\mathcal{A}(\mathcal{H}_{\Lambda})\hat{\otimes}B)^{\Lambda}),\\
    \alpha_{*}:K_{*}(C^{*}_{\max}(\Lambda,\mathcal{A}(\mathcal{H}_{\Lambda})\hat{\otimes}B)^{\Lambda}) \to K_{*}(C^{*}_{\max}(\Lambda, \mathcal{S}\hat{\otimes}B)^{\Lambda}).
  \end{aligned}
  $$

Define the Hilbert module $E_\phi = \mathcal{H}_\phi \hat{\otimes} \mathcal{A}(\mathcal{H}_{\Lambda})$ which carries a diagonal action of $\Lambda$. Since we have the isomorphism $\mathbb{C}[\Lambda, B \hat{\otimes} \mathcal{A}(\mathcal{H}_{\Lambda})]^\Lambda\cong C_c(\Lambda, B \hat{\otimes} \mathcal{A}(\mathcal{H}_{\Lambda})\hat\otimes\mathcal K) \cong C_c(\Lambda, B\hat\otimes\mathcal K) \hat\otimes \mathcal{A}(\mathcal{H}_{\Lambda}),$
the algebra $\mathbb{C}[\Lambda, B \hat{\otimes} \mathcal{A}(\mathcal{H}_{\Lambda}))\cong \mathcal{A}(\mathcal{H}_{\Lambda})]^\Lambda$ admits a natural representation on $E_\phi$. This representation, in turn, determines the intermediate $C^*$-completion $C^{*}_{\phi}(\Lambda,\mathcal{A}(\mathcal{H}_{\Lambda})\hat{\otimes}B)^{\Lambda}$.

  The assumption that $\phi$ lies between the maximal and reduced crossed products is equivalent to the norm inequalities $\|\cdot\|_r \le \|\cdot\|_\phi \le \|\cdot\|_{\max}$. These inequalities are preserved for the extended representation $\phi_{\mathcal{A}}$. Thus, the associated completion $C^{*}_{\phi}(\Lambda,\mathcal{A}(\mathcal{H}_{\Lambda})\hat{\otimes}B)^{\Lambda}$ is also sandwiched between the maximal and reduced versions. Consequently, we have that
  $$
  \lambda_{\mathcal{A}}: C^{*}_{\max}(\Lambda,\mathcal{A}(\mathcal{H}_{\Lambda})\hat{\otimes}B)^{\Lambda} \xrightarrow{\pi_{\phi,\mathcal{A}}} C^{*}_{\phi}(\Lambda,\mathcal{A}(\mathcal{H}_{\Lambda})\hat{\otimes}B)^{\Lambda} \xrightarrow{\rho_{\phi,\mathcal{A}}} C^{*}(\Lambda,\mathcal{A}(\mathcal{H}_{\Lambda})\hat{\otimes}B)^{\Lambda},
  $$
  where $\lambda_{\mathcal{A}} = \rho_{\phi,\mathcal{A}} \circ \pi_{\phi,\mathcal{A}}$.
  
  Similarly, we can define the Bott map and Dirac map for this intermediate completion:
  $$
  \begin{aligned}
   \beta_{\phi,*}:K_{*}(C^{*}_{\phi}(\Lambda, B)^{\Lambda} \hat{\otimes}\mathcal{S}) \to K_{*}(C^{*}_{\phi}(\Lambda,\mathcal{A}(\mathcal{H}_{\Lambda})\hat{\otimes}B)^{\Lambda}),\\
    \alpha_{\phi,*}:K_{*}(C^{*}_{\phi}(\Lambda,\mathcal{A}(\mathcal{H}_{\Lambda})\hat{\otimes}B)^{\Lambda}) \to K_{*}(C^{*}_{\phi}(\Lambda, \mathcal{S}\hat{\otimes}B)^{\Lambda}).
  \end{aligned}
  $$
By definition, we have the following commutative diagrams on $K$-theory:
  \begin{equation}\label{cmmu-atmenable}
    \begin{tikzcd}
	{K_{*}(C^{*}_{\max}(\Lambda, B)^{\Lambda} \hat{\otimes}\mathcal{S})} & {K_{*}(C^{*}_{\phi}(\Lambda, B)^{\Lambda} \hat{\otimes}\mathcal{S})} & {K_{*}(C^{*}(\Lambda, B)^{\Lambda} \hat{\otimes}\mathcal{S})} \\
	{K_{*}(C^{*}_{\max}(\Lambda,\mathcal{A}(\mathcal{H}_{\Lambda})\hat{\otimes}B)^{\Lambda})} & {K_{*}(C^{*}_{\phi}(\Lambda,\mathcal{A}(\mathcal{H}_{\Lambda})\hat{\otimes}B)^{\Lambda})} & {K_{*}(C^{*}(\Lambda,\mathcal{A}(\mathcal{H}_{\Lambda})\hat{\otimes}B)^{\Lambda})} \\
	{K_{*}(C^{*}_{\max}(\Lambda, B)^{\Lambda} \hat{\otimes}\mathcal{S})} & {K_{*}(C^{*}_{\phi}(\Lambda, B)^{\Lambda} \hat{\otimes}\mathcal{S})} & {K_{*}(C^{*}(\Lambda, B)^{\Lambda} \hat{\otimes}\mathcal{S})}
	\arrow["{(\pi_{\phi})_{*}}"', from=1-1, to=1-2]
	\arrow["{\beta_{*}}"', from=1-1, to=2-1]
	\arrow["{(\rho_{\phi})_{*}}"', from=1-2, to=1-3]
	\arrow["{\beta_{\phi,*}}"', from=1-2, to=2-2]
	\arrow["{\beta_{*}}", from=1-3, to=2-3]
	\arrow["{(\pi_{\phi,\mathcal{A}})_{*}}"', from=2-1, to=2-2]
	\arrow["{\alpha_{*}}"', from=2-1, to=3-1]
	\arrow["{(\rho_{\phi,\mathcal{A}})_{*}}"', from=2-2, to=2-3]
	\arrow["{\alpha_{\phi,*}}"', from=2-2, to=3-2]
	\arrow["{\alpha_{*}}", from=2-3, to=3-3]
	\arrow["{(\pi_{\phi})_{*}}"', from=3-1, to=3-2]
	\arrow["{(\rho_{\phi})_{*}}"', from=3-2, to=3-3]
\end{tikzcd}
  \end{equation}

  Since the group $\Lambda$ acts on $\mathcal{H}_{\Lambda}$ properly, the algebra $C^{*}_{\max}(\Lambda,\mathcal{A}(\mathcal{H}_{\Lambda})\hat{\otimes}B)^{\Lambda}$ is isomorphic to $C^{*}(\Lambda,\mathcal{A}(\mathcal{H}_{\Lambda})\hat{\otimes}B)^{\Lambda}$, see \cite{Tu-1999}. In other words, the canonical quotient map $\lambda_{\mathcal{A}}$ is an isomorphism. As a result, both $\pi_{\phi,\mathcal{A}}$ and $\rho_{\phi,\mathcal{A}}$ are isomorphisms.

  By the same discussion in Proposition \ref{2st thm}, the compositions $\alpha_{\phi,*} \circ \beta_{\phi,*}$ and $\alpha_{*} \circ \beta_{*}$ are all identity maps on $K$-theory. Since the horizontal maps in the middle row $(\pi_{\phi,\mathcal{A}})_{*}$ and $(\rho_{\phi,\mathcal{A}})_{*}$ are isomorphisms, it follows from the commutative diagram that the corresponding maps in the top row, $(\pi_{\phi})_{*}$ and $(\rho_{\phi})_{*}$, must also be isomorphisms.
\end{proof}

\begin{lemma}\label{cut-lemma}
Let $x=(t,v)$ be a point in $\mathbb{R}_{+}\times \mathcal{H}$, $\Gamma_{x}\le \Gamma$ the stablizer of $x$. Then there exists a $\Gamma_x$-invariant neighborhood $U$ of $x$ such that $\Gamma\cdot U$ is homeomorphic the balanced product $\Gamma\times_{\Gamma_x}U$.
\end{lemma}

\begin{proof}
The space $\mathbb{R}_{+}\times \mathcal{H}$ admits a diagonal action of $\Gamma$ defined by $\gamma \cdot (t,v) = (t, \gamma\cdot v)$ for all $\gamma \in \Gamma$ and $(t,v) \in \mathbb{R}_{+}\times \mathcal{H}$. The stabilizer $\Gamma_{x}$ of the point $x=(t,v)$ coincides with the stabilizer $\Gamma_{v}$ of the vector $v$ in $\mathcal{H}$. By Corollary \ref{keycoro}, we show that $\Gamma_{x}=\Gamma_{v}$ is commensurable with $\Lambda$.
  
First, the orbit $\Gamma \cdot x$ is discrete. Indeed, $\Gamma\cdot x$ can be identified with $\Gamma/\Gamma_x$. For any neighbourhood $W_1\subseteq U \subseteq \mathbb R_+\times \mathcal H$ of $x$ with compact closure, there exists a sufficiently large $R>0$ and some $(0,v)\in \mathbb R_+\times \mathcal H$ such that $B((0,v),R)$ covers $\overline{W_1}$. By Proposition \ref{Keyprop}, $W_1\cap \Gamma\cdot x$ must be finite. Thus, we can take $W_2\subseteq W_1$ such that $g W_2\cap W_2\ne\emptyset$ if and only if $g\in\Gamma_x$. Such $W_2$ exists by equivariant coarse embeddability of $X$ and Proposition \ref{Keyprop}. Define $W_3=\bigcup_{g\in\Gamma_x}gW_2$. We then claim that $W_3\cap gW_3=\emptyset$ for any $g\notin \Gamma_x$. To see this, for any $g\notin \Gamma_x$, assume for a contradiction that $gW_3\cap W_3\ne \emptyset$. By definition, there exists $h_1, h_2\in\Gamma_x$ such that $gh_1W_2\cap h_2W_2\ne \emptyset$. Equivalently speaking, $h_2^{-1}gh_1W_2\cap W_2\ne \emptyset$. Since $g\notin \Gamma_x$, we conclude that $h_2^{-1}gh_1\notin\Gamma_x$. This leads to a contradiction to the fact that $gW_2\cap W_2$ is nonempty if and only if $g\in \Gamma_x$.

From the definition of $W_3$, it is direct to see that $W_3$ is $\Gamma_x$-invariant. Thus, $\Gamma\cdot W_3$ is homeomorphic to $\Gamma\times_{\Gamma_x}W_3$ since $g W_3\cap W_3=\emptyset$ for any $g\in\Gamma\setminus\Gamma_x$.  Taking $U=W_3$, we then finish the proof.
\end{proof}

\begin{remark}
An alternative approach to the above lemma involves the Schlichting completion $(G, H)$ of the pair $(\Gamma,\Lambda)$. As established in \cite{Clement-2023}, the quotient space $\Gamma/\Lambda$ admits a $\Gamma$-equivariant coarse embedding into a Hilbert space if and only if its Schlichting completion $G$ is a-T-menable. Consequently, $G$ admits a proper affine isometric action on a Hilbert space. By the Slice Lemma for proper actions (see, for instance, \cite[Lemma A.2.7]{willett2020higher}), there exists a $G$-slice for the action.  Restricting this action to $\Gamma$, we obtain a corresponding slice for the $\Gamma$-action that satisfies the requirements of Lemma \ref{cut-lemma}.
\end{remark}

From Lemma \ref{cut-lemma}, the ball $O_{r}(e)$ admits a finite open cover $\{ V_{i} \}_{i=1}^{n}$ with corresponding subgroups $\{ N_{i} \}_{i=1}^{n}$ satisfying the properties stated in the lemma. Consequently, $O_{r}$ admits the open cover $\{ \Gamma\cdot V_{i} \}_{i=1}^{n}$.

\begin{theorem}
  \label{1st thm}
  Let $(\Gamma, \Lambda)$ be a Hecke pair and $X = \Gamma/\Lambda$. If $X$ admits a $\Gamma$-equivariant coarse embedding into a Hilbert space and $\Lambda$ is a-T-menable, then for any $\Gamma$-$C^*$-algebra $B$, the canonical quotient map induces an isomorphism in $K$-theory 
  $$
  \lambda_{*}:K_{*}(C_{\max}^{*}(\Gamma,\mathcal{A}(\mathcal{H})\hat{\otimes}B)^{\Gamma} ) \to K_{*}(C^{*}(\Gamma,\mathcal{A}(\mathcal{H})\hat{\otimes}B)^{\Gamma} ).
  $$
\end{theorem}

\begin{proof}
Write $A_{U}=\mathcal{A}(\mathcal{H})_{U}\hat{\otimes}B$ for short, where $U$ is an open subset of $\mathbb{R}_{+}\times \mathcal{H}$.

For any $r < r'$, recall that the reduced twisted Roe algebra $C^{*}(\Gamma,A_{O_{r}})^{\Gamma}$ is concretely defined as a subalgebra of operators on the Hilbert module $E_{r} = \ell^2(\Gamma) \hat{\otimes} \mathcal{A}(\mathcal{H})_{O_{r}} \hat{\otimes} B \hat{\otimes} \mathcal{K}$. Since $\mathcal{A}(\mathcal{H})_{O_{r}}$ is a subalgebra of $\mathcal{A}(\mathcal{H})_{O_{r'}}$, the Hilbert module $E_{r}$ can be naturally viewed as a submodule of $E_{r'}$. It is direct to see that the reduced norm via the left regular representation is consistent with this inclusion. That is, for any operator $T \in \mathbb{C}[\Gamma,A_{O_{r}}]^{\Gamma}$, its norm as an operator on $E_{r}$ is identical to its norm as an operator on $E_{r'}$. Consequently, we have an isometric embedding
  $$
  \iota_{\mathrm{red}}^{r,r'}: C^*(\Gamma,A_{O_{r}})^{\Gamma} \hookrightarrow C^*(\Gamma,A_{O_{r'}})^{\Gamma}.
  $$
It then follows from Definition \ref{defalgebraictwisted} that $C^*(\Gamma,\mathcal{A}(\mathcal{H})\hat{\otimes}B)^{\Gamma} = \varinjlim C^*(\Gamma,A_{O_{r}})^{\Gamma}$. As $K$-theory commutes with inductive limits, we have
  $$
  K_{*}(C^{*}(\Gamma,\mathcal{A}(\mathcal{H})\hat{\otimes}B)^{\Gamma} ) \cong \lim_{r \to \infty} K_{*}(C^{*}(\Gamma,A_{O_{r}})^{\Gamma} ).
  $$
  
For the maximal norm, for any $r<r'$ and any operator $T$ in the algebraic algebra $\mathbb{C}[\Gamma,A_{O_{r}}]^{\Gamma}$, its norm in $C^{*}_{\max}(\Gamma,A_{O_{r'}})^{\Gamma}$, denoted by $\|T\|_{\max, O_{r'}}$, is defined to be the supremum of all $*$-representations of $\mathbb{C}[\Gamma,A_{O_{r'}}]^{\Gamma}$. Since any such representation restricts to a representation of the subalgebra $\mathbb{C}[\Gamma,A_{O_{r}}]^{\Gamma}$, the set of representations defining $\|T\|_{\max, O_{r'}}$ is a subset of those defining $\|T\|_{\max, O_{r}}$. Consequently, we have the norm inequality:
  $$ 
  \|T\|_{\max, O_{r'}} \le \|T\|_{\max, O_{r}}. 
  $$
Thus, we denote $\phi$ the $*$-representation of $\mathbb{C}[\Gamma,A_{O_{r}}]^{\Gamma}$ on some $\mathcal{H}_{\phi}$ the restriction of the universal representation of $\mathbb{C}[\Gamma,A_{O_{r'}}]^{\Gamma}$, and denote $C^{*}_{\phi}(\Gamma,A_{O_{r}})^{\Gamma}$ the completion under the norm $\|\cdot \|_{\mathcal{B}(\mathcal{H}_{\phi})}$. Then, for any $T \in \mathbb{C}[\Gamma,A_{O_{r}}]^{\Gamma}$, we have $\|T\|_{\mathcal{B}(\mathcal{H}_{\phi})} = \|\iota^{r,r'}(T)\|_{\max, r'}$, where $\iota^{r,r'}$ denotes the natural inclusion 
  $$
  \iota^{r,r'}: \mathbb{C}[\Gamma,A_{O_{r}}]^{\Gamma} \hookrightarrow \mathbb{C}[\Gamma,A_{O_{r'}}]^{\Gamma}.
  $$
  By the universal property of maximal completions, there exists a $*$-homomorphism $\pi_{\phi}$ from $C^{*}_{\max}(\Gamma,A_{O_{r}})^{\Gamma}$ to $C^{*}_{\phi}(\Gamma,A_{O_{r}})^{\Gamma}$. Therefore, we obtain a $*$-homomorphism
  $$
  \iota_{\max}^{r,r'}: C^{*}_{\max}(\Gamma,A_{O_{r}})^{\Gamma} \xrightarrow{\pi_{\phi}} C^{*}_{\phi}(\Gamma,A_{O_{r}})^{\Gamma} \hookrightarrow C^{*}_{\max}(\Gamma,A_{O_{r'}})^{\Gamma}.
  $$
It follows that $C^{*}_{\max}(\Gamma,\mathcal{A}(\mathcal{H})\hat{\otimes}B)^{\Gamma} = \varinjlim C^{*}_{\max}(\Gamma,A_{O_{r}})^{\Gamma}$ under the $*$-homomorphisms $\{\iota^{r,r'}\}_{r'>r>0}$. As $K$-theory commutes with inductive limits, we have
  $$
  K_{*}(C_{\max}^{*}(\Gamma,\mathcal{A}(\mathcal{H})\hat{\otimes}B)^{\Gamma} ) \cong \lim_{r \to \infty} K_{*}(C_{\max}^{*}(\Gamma,\mathcal{A}(\mathcal{H})_{O_{r }}\hat{\otimes}B)^{\Gamma} ).
  $$

  Therefore, to prove the theorem, it suffices to show that for each $r>0$, the map
  $$
  (\lambda_{r})_*:K_{*}(C_{\max}^{*}(\Gamma,A_{O_{r}})^{\Gamma} ) \to K_{*}(C^{*}(\Gamma,A_{O_{r}})^{\Gamma} )
  $$
  is an isomorphism. 

  By the Lemma \ref{cut-lemma}, the set $O_{r}$ admits a finite open cover $\{ \Gamma\cdot V_{i} \}_{i=1}^{n}$ such that each $\Gamma\cdot V_{i}$ is homeomorphic to a balanced product $\Gamma\times_{N_i}V_i$, where each $N_i$ is a subgroup of $\Gamma$ commensurable with $\Lambda$. Using the Mayer--Vietoris sequence and the Five Lemma, we can reduce the problem to showing that for each $i$, the map
  $$  
  (\lambda_{V_i})_*:K_{*}(C^{*}_{\max}(\Gamma,A_{O_{r}})^{\Gamma} \mid_{\Gamma\cdot V_{i}} ) \to K_{*}(C^{*}(\Gamma,A_{O_{r}})^{\Gamma}\mid_{\Gamma\cdot V_{i}} ) 
  $$
  is an isomorphism. Here the notation $\mid_{\Gamma\cdot V_{i}}$ signifies restricting the support of operators to $\Gamma \times \Gamma \times (\mathbb{R}_{+}\times \Gamma\cdot V_{i})$ within the respective norms. This is essentially the completion of the algebraic equivariant Roe algebra $\mathbb{C}[\Gamma,A_{\Gamma \cdot V_{i}}]^{\Gamma}$ under the norms induced from $C_{\max}^{*}(\Gamma,A_{O_{r}})^{\Gamma}$ and $C^{*}(\Gamma,A_{O_{r}})^{\Gamma}$, respectively. Similarly, for reduced norms, we have an isomorphism
  $$
  C^{*}(\Gamma,A_{O_{r}})^{\Gamma}\mid_{\Gamma\cdot V_{i}} \cong C^{*}(\Gamma,A_{\Gamma \cdot V_{i}})^{\Gamma}.
  $$
  For the maximal norms, similar to above, we denote the restriction of the universal representation of $\mathbb{C}[\Gamma,A_{O_{r}}]^{\Gamma}$ by
$$\psi: \mathbb{C}[\Gamma,A_{\Gamma \cdot V_{i}}]^{\Gamma}\to \mathcal{B}(\mathcal{H}_{\psi}),$$
and denote by $C^{*}_{\psi}(\Gamma,A_{\Gamma \cdot V_{i}})^{\Gamma}$ the completion under the norm $\|\cdot \|_{\mathcal{B}(\mathcal{H}_{\psi})}$. Since the reduced norm is hereditary to subalgebras, there exists a canonical quotient
$$\rho_{\psi}: C^{*}_{\psi}(\Gamma,A_{\Gamma \cdot V_{i}})^{\Gamma}\cong C^{*}_{\max}(\Gamma,A_{O_{r}})^{\Gamma} \mid_{\Gamma\cdot V_{i}}  \to C^{*}(\Gamma,A_{\Gamma \cdot V_{i}})^{\Gamma},$$
By the universal property of maximal completions, there exists a canonical quotient map
$$\pi_{\psi}: C^{*}_{\max}(\Gamma,A_{\Gamma \cdot V_{i}})^{\Gamma}\to C^{*}_{\psi}(\Gamma,A_{\Gamma \cdot V_{i}})^{\Gamma}\cong C^{*}_{\max}(\Gamma,A_{O_{r}})^{\Gamma} \mid_{\Gamma\cdot V_{i}} .$$
Combining these two maps, we have that
  $$
  \lambda_{V_i}:C^{*}_{\max}(\Gamma,A_{\Gamma \cdot V_{i}})^{\Gamma} \xrightarrow{\pi_{\psi}} C^{*}_{\psi}(\Gamma,A_{\Gamma \cdot V_{i}})^{\Gamma} \xrightarrow{\lambda_{i}'}  C^{*}(\Gamma,A_{\Gamma \cdot V_{i}})^{\Gamma}.
  $$

    Since $\Gamma \cdot V_i \cong \Gamma \times_{N_i} V_i$, the algebra $A_{\Gamma \cdot V_i}$ is isomorphic to the induced algebra $\operatorname{Ind}^{\Gamma}_{N_i} (A_{V_i})$. According to Propositions \ref{reduced-crossed-product} and Proposition \ref{maximal-crossed-product} and Green's Imprimitivity Theorem (see \cite[Theorem 4.22]{Williams2007} for example), we have the following isomorphisms:
$$K_{*}(C^{*}_{\max}(\Gamma, A_{\Gamma \cdot V_i})^{\Gamma} ) \cong K_{*}(A_{\Gamma\cdot V_{i}}\rtimes _{\max} \Gamma)\cong K_{*}(A_{V_{i}}\rtimes _{\max} N_{i}),$$
$$K_{*}(C^{*}_{\psi}(\Gamma, A_{\Gamma \cdot V_i})^{\Gamma} ) \cong K_{*}(A_{\Gamma\cdot V_{i}}\rtimes _{\psi} \Gamma)\cong K_{*}(A_{V_{i}}\rtimes _{\psi} N_{i}),$$
$$K_{*}(C^{*}(\Gamma, A_{\Gamma \cdot V_i})^{\Gamma} ) \cong K_{*}(A_{\Gamma\cdot V_{i}}\rtimes_{\mathrm{r}} \Gamma)\cong K_{*}(A_{V_{i}}\rtimes_{\mathrm{r}} N_{i}).$$
  Therefore, the canonical quotient map $\lambda_{V_i}$ induces a map
  $$
  (\lambda_{V_i})_{*}: K_{*}(A_{V_i} \rtimes_{\max} N_i) \xrightarrow{(\pi_{\psi})_{*}}  K_{*}(A_{V_i} \rtimes_{\psi} N_i) \xrightarrow{(\lambda_{i}')_{*}} K_{*}(A_{V_i} \rtimes_{\mathrm{r}} N_i).
  $$

Recall that the Haagerup property is invariant under commensurability. Specifically, it is hereditary to subgroups, and conversely, the a-T-menability of a finite index subgroup implies the a-T-menability of the ambient group, \cite[Proposition 6.1.5]{CCJJV-2001}. This allows us to deduce the property for $\Gamma$ from its commensurable subgroups.
Thus, $N_{i}$ is a-T-menable for each $i$. By Lemma \ref{lem: sandwich}, the map $(\lambda_{V_i})_{*}$ is an isomorphism. This completes the proof. 
\end{proof}

\subsection{The proof of the main theorem}

Now we are in a position to prove the Theorem \ref{A}. 
\begin{theorem}\label{main theorem}
  Let $(\Gamma,\Lambda)$ be a Hecke pair of discrete groups. Assume that $\Lambda$ is a-T-menable and that the quotient space $X = \Gamma/\Lambda$ admits a $\Gamma$-equivariant coarse embedding into Hilbert space. Then for every $\Gamma$-$C^{*}$-algebra $B$ the canonical map
  \begin{equation}\label{main-thm-iso}
    K_{*}(C^{*}_{\max}(\Gamma,B)^{\Gamma}) \longrightarrow K_{*}(C^{*}(\Gamma,B)^{\Gamma})
  \end{equation}
is an isomorphism.
\end{theorem}
\begin{proof}
  We have constructed the following commutative diagram:
  \begin{equation}\label{cmmu2}
    \begin{tikzcd}
	{   K_{*}(C^{*}_{\max}(\Gamma, B)^{\Gamma} \hat{\otimes}\mathcal{S})} && {   K_{*}(C^{*}(\Gamma, B)^{\Gamma} \hat{\otimes}\mathcal{S})} \\
	{ K_{*}(C^{*}_{\max}(\Gamma,\mathcal{A}(\mathcal{H})\hat{\otimes}B)^{\Gamma} )} && { K_{*}(C^{*}(\Gamma,\mathcal{A}(\mathcal{H})\hat{\otimes}B)^{\Gamma} )} \\
	{   K_{*}(C^{*}_{\max}(\Gamma, B)^{\Gamma} \hat{\otimes}\mathcal{S})} && {   K_{*}(C^{*}(\Gamma, B)^{\Gamma} \hat{\otimes}\mathcal{S})}
	\arrow["{{\lambda_{*}}}", from=1-1, to=1-3]
	\arrow["{{\beta_{*}}}"', from=1-1, to=2-1]
	\arrow["{{\beta_{*}}}", from=1-3, to=2-3]
	\arrow["{{\lambda_{*}}}", from=2-1, to=2-3]
	\arrow["{\alpha_{*}}"', from=2-1, to=3-1]
	\arrow["{\alpha_{*}}", from=2-3, to=3-3]
	\arrow["{\lambda_{*}}", from=3-1, to=3-3]
\end{tikzcd}
  \end{equation}

  As we can see in the commutative diagram \eqref{cmmu2}, the middle horizontal map $\lambda_{*}$ is an isomorphism by Theorem \ref{1st thm}. Since the composition of vertical maps $\alpha_{*} \circ \beta_{*}$ is the identity by Proposition \ref{2st thm}, the main theorem follows immediately.
\end{proof}

As a direct consequence, we obtain the unconditional result for Hecke pairs where $\Lambda$ is a-T-menable, as discussed in the introduction.

\begin{corollary}
Let $1\to \Lambda\to \Gamma\to \Gamma/\Lambda\to 1$ be an extension of countable discrete groups. If both $\Lambda$ and $\Gamma/\Lambda$ are a-T-menable, then for any $\Gamma$-$C^*$-algebra $B$, the canonical quotient map induces an isomorphism in $K$-theory
  $$
  \lambda_*: K_*(B \rtimes \Gamma) \xrightarrow{\cong} K_*(B \rtimes_{r} \Gamma).
  $$
\end{corollary}

\begin{proof}
Since $\Lambda$ is normal, the pair $(\Gamma, \Lambda)$ is a Hecke pair. Moreover, $\Gamma/\Lambda$ admits a proper affine isometric action on $\alpha:\Gamma/\Lambda\curvearrowright\mathcal H$. Set $b: \Gamma/{\Lambda}\to\mathcal H$ to be the $1$-cocycle associated with $\alpha$. We then define $\widetilde{\alpha}: \Gamma\curvearrowright\mathcal H$ by $\widetilde{\alpha}_{g}=\alpha_{g\Lambda}$ for any $g\in\Gamma$, where $g\Lambda\in\Gamma/\Lambda$ is the coset determined by $g$. Then
$$h:\Gamma/\Lambda\to\mathcal H,\qquad\text{defined by } g\Lambda\mapsto b_{g\Lambda}$$
is a $\Gamma$-equivariant coarse embedding associated with $\widetilde\alpha$.
\end{proof}

\section{The Novikov conjecture and the coarse Baum--Connes conjecture}


In Section \ref{section4}, we established the $K$-theoretic isomorphism \eqref{main-thm-iso} using the Dirac-dual-Dirac method. We observe that this approach can be adapted to address the Baum--Connes conjecture for Hecke pairs. In this section, rather than repeating the technical details, we briefly summarize the necessary geometric constructions (specifically, the localization algebras on the Rips complex) and outline the proof of the conjecture, utilizing the framework established in the previous section. We remark that the Baum--Connes conjecture for Hecke pairs in this geometric setting has been previously established by Dell'Aiera in \cite{Clement-2023} using different methods. Our goal in this section is to provide an alternative proof based explicitly on the Dirac-dual-Dirac method.

\subsection{Applications to the Baum--Connes conjecture}
The Rips complex $P_{d}(\Gamma)$ is a simplicial complex with vertex set $\Gamma$ and simplices of diameter $\le d$. As a proper, contractible metric space with bounded geometry, $P_d(\Gamma)$ admits a proper and cocompact $\Gamma$-action. We define the localization algebra $C^{*}_{L}(P_{d}(\Gamma), B)^{\Gamma}$ as the completion of bounded, uniformly continuous functions $g: [0,\infty) \to \mathbb{C}[P_d(\Gamma), B]^{\Gamma}$ such that $\operatorname{prop}(g(t))\to 0$ as $t\to\infty$, equipped with the supremum norm. The evaluation map $e(g)=g(0)$ induces the coarse assembly map:
$$
\mu: \lim_{d \to \infty} K_{*}(C^{*}_{L}(P_{d}(\Gamma), B)^{\Gamma}) \to \lim_{d \to \infty} K_{*}(C^{*}(P_{d}(\Gamma), B)^{\Gamma}) \cong K_{*}(C^{*}(\Gamma, B)^{\Gamma}).
$$
The isomorphism on the right follows from the $\Gamma$-cocompactness of $P_d(\Gamma)$, which ensures that the equivariant Roe algebra is Morita equivalent to the reduced crossed product $B \rtimes_r \Gamma$.

To implement the Dirac-dual-Dirac argument, we extend the construction to twisted versions. Fix a countable dense $\Gamma$-equivariant subset $Z_d \subset P_{d}(\Gamma)$ and a $\Gamma$-equivariant map $J: Z_d \to \Gamma$ determined by a bounded fundamental domain $\Delta_d$. The twisted Roe algebra $C^{*}(P_{d}(\Gamma), \mathcal{A}(\mathcal{H})\hat{\otimes}B)^{\Gamma}$ and its localization version $C^{*}_{L}$ are defined by replacing the index set $\Gamma$ with $Z_d$ in the algebraic twisted Roe algebra $\mathbb{C}[P_{d}(\Gamma), \mathcal{A}(\mathcal{H})\hat{\otimes}B]^{\Gamma}$. Crucially, the condition (4) in Definition \ref{defalgebraictwisted} is replaced by the support condition $\operatorname{supp}(T_{y,z}) \subseteq O_{r}(J(y))$ for all $y,z \in Z_d$.

The following lemma, which is a direct analogue of \cite[Lemma 5.5]{MR4422220}, allows for a local reduction of the assembly map via a finite open cover $\{ \Gamma \cdot V_{i} \}_{i=1}^n$ of the orbit $O_r$.

\begin{lemma}\label{decompose}
The inclusion-induced maps 
$$
(\iota_i)_*: K_*(C^*(P_d(N_i), \mathcal{A}_{V_i}\hat{\otimes}B)^{N_i}) \to K_*(C^*(P_d(\Gamma), \mathcal{A}\hat{\otimes}B)_{\Gamma \cdot V_i}^\Gamma)
$$
and its localization version $(\iota_{L,i})_{*}$ are isomorphisms.
\end{lemma}

The Bott map $\beta$ and Dirac map $\alpha$ extend naturally to the Rips complex setting by substituting $P_d(\Gamma)$ for $\Gamma$ and $W(J(y))$ for $W(x)$. Since $\varinjlim_d P_d(\Gamma)$ models the universal space $\underline{E}\Gamma$, the direct limit of localization algebras recovers the topological side $K_{*}^{\Gamma}(\underline{E}\Gamma, B)$ \cite{Yu-1995, Yu-2000}. We thus arrive at our main result.

\begin{theorem}\label{BCC}
   Let $(\Gamma, \Lambda)$ be a Hecke pair of discrete groups. If $\Lambda$ is a-T-menable and the quotient space $X=\Gamma /\Lambda$ admits a $\Gamma$-equivariant coarse embedding into Hilbert space, then $\Gamma$ satisfies the Baum--Connes conjecture with coefficients in any $\Gamma$-$C^{*}$-algebra $B$.
\end{theorem}

\begin{proof}
  The above construction yields the commutative diagram
  \begin{equation}\label{cmmu3}
    \begin{tikzcd}
	{\lim\limits_{d\to\infty}K_{*}(C_{L}^{*}(P_{d}(\Gamma),B)^{\Gamma}\hat{\otimes}\mathcal{S})} && {\lim\limits_{d\to\infty}K_{*}(C^{*}(P_{d}(\Gamma),B)^{\Gamma}\hat{\otimes}\mathcal{S})} \\
	{\lim\limits_{d\to\infty}K_{*}(C_{L}^{*}(P_{d}(\Gamma),\mathcal{A}(\mathcal{H})\hat{\otimes}B)^{\Gamma})} && {\lim\limits_{d\to\infty}K_{*}(C^{*}(P_{d}(\Gamma),\mathcal{A}(\mathcal{H})\hat{\otimes}B)^{\Gamma})} \\
	{\lim\limits_{d\to\infty}K_{*}(C_{L}^{*}(P_{d}(\Gamma),B)^{\Gamma}\hat{\otimes}\mathcal{S})} && {\lim\limits_{d\to\infty}K_{*}(C^{*}(P_{d}(\Gamma),B)^{\Gamma}\hat{\otimes}\mathcal{S})}
	\arrow[from=1-1, to=1-3]
	\arrow["{\beta_{*}}"', from=1-1, to=2-1]
	\arrow["{\beta_{*}}", from=1-3, to=2-3]
	\arrow["\mu_{\mathcal{A}}", from=2-1, to=2-3]
	\arrow["{\alpha_{*}}"', from=2-1, to=3-1]
	\arrow["{\alpha_{*}}", from=2-3, to=3-3]
	\arrow[from=3-1, to=3-3]
\end{tikzcd}
  \end{equation}

By Lemma \ref{decompose}, the isomorphism of the twisted assembly map $\mu_{\mathcal{A}}$ reduces to that of the local maps $\mu_{\mathcal{A}, N_i}$. Each $N_i$ is commensurable with $\Lambda$ and thus a-T-menable, which implies that $\mu_{\mathcal{A}, N_i}$ (and hence $\mu_{\mathcal{A}}$) is an isomorphism. Given the identity $\alpha_{*} \circ \beta_{*} = \text{id}$ \cite{Yu-2000}, the horizontal map $\mu$ is an isomorphism.
\end{proof}

\subsection{Applications to the strong Novikov conjecture}

In addition to the Dirac-dual-Dirac method, one can also employ the results of \cite{MR4422220} to establish the strong Novikov conjecture for $\Gamma$ with coefficients in any $\Gamma$-$C^{*}$-algebra $B$. The precise statement is as follows.
\begin{theorem}\label{strongNovikov}
  Let $(\Gamma, \Lambda)$ be a Hecke pair of discrete groups. If $\Lambda$ admits a coarse embedding into Hilbert space, and the quotient space $X=\Gamma /\Lambda$ admits a $\Gamma$-equivariant coarse embedding into Hilbert space, then $\Gamma$ satisfies the strong Novikov conjecture with coefficients in any $\Gamma$-$C^{*}$-algebra $B$.
\end{theorem}

To prove this, we construct a $\Gamma$-space $W$ such that for every subgroup $N \le \Gamma$ commensurable with $\Lambda$, the groupoid $W \rtimes N$ is a-T-menable, and for every finite $F \le \Gamma$, $W$ is $F$-contractible. 

Fix such an $N$ and pick coset representatives $S_N$ so that $\Gamma=\bigsqcup_{s\in S_N}sN$. Since $N$ admits a coarse embedding $h_N': N\to\mathcal{H}_0$, we extend it to $h_N: \Gamma\to\mathcal{H}_0$ by $h_N(sn)=h_N'(n)$. For a fixed $n\in N$, define the bounded function $f_n: \Gamma\to\mathbb{R}$ by
$$
f_n(\gamma)=\|h_N(\gamma n)-h_N(\gamma)\|^2.
$$
Let $A$ be the unital, commutative $\Gamma$-invariant $C^{*}$-subalgebra of $\ell^{\infty}(\Gamma)$ generated by all $f_n$, their right translates, $C_0(\Gamma)$, and constants. The spectrum $W'_N$ of $A$ is a compactification of $\Gamma$. Let $W_N$ be the space of probability measures on $W'_N$ equipped with the weak-$*$ topology. Considering the collection $S_\Lambda$ of all subgroups commensurable with $\Lambda$, we define the product space $W := \prod_{N \in S_\Lambda} W_N$ and equip it with the $\Gamma$-action $\gamma \cdot (w_N) = (w_{\gamma^{-1} N \gamma})$. The following results are established in \cite[Section 4.1]{MR4422220}. 
\begin{proposition}
For every $N\in S_\Lambda$, the transformation groupoid $W\rtimes N$ is a-T-menable. Moreover, for any finite subgroup $F\le\Gamma$, the space $W$ is $F$-contractible.
\end{proposition}

The equivariant Roe and localization algebras with coefficients $C(W)\hat{\otimes}B$ are defined as in Section \ref{section4}. Tu \cite{Tu-groupoid} proved that for any a-T-menable groupoid $X\rtimes G$, the assembly map is an isomorphism for any $G$-algebra $D$. Combined with Lemma \ref{decompose}, we obtain the following result:
\begin{lemma}\label{isomorphism with W}
For each $d>0$, there is an isomorphism
$$
\mu_{\mathcal{A},W}: K_{*}(C^{*}_{L}(P_{d}(\Gamma),C(W)\hat{\otimes}\mathcal{A}(\mathcal{H})\hat{\otimes}B)^{\Gamma}) \to K_{*}(C^{*}(P_{d}(\Gamma),C(W)\hat{\otimes}\mathcal{A}(\mathcal{H})\hat{\otimes}B)^{\Gamma}).
$$
\end{lemma}

The Bott map $\beta_{*}$ defined in \eqref{cmmu3} can extends to coefficients with $C(W)$. Parallel to \cite[Proposition 6.4]{MR4422220}, we have:
\begin{lemma}\label{localization iso}
For each $d>0$, the Bott map 
\begin{equation}\label{bott-cw-ah-L}
\beta_{L,*}: K_{*}(C^{*}_{L}(P_{d}(\Gamma),C(W)\hat{\otimes}B)^{\Gamma}\hat{\otimes}\mathcal{S}) \to K_{*}(C^{*}_{L}(P_{d}(\Gamma),C(W)\hat{\otimes}\mathcal{A}(\mathcal{H})\hat{\otimes}B)^{\Gamma})
\end{equation}
is an isomorphism.
\end{lemma}
\begin{proof}[Sketch of proof]
We proceed by induction on the dimension of the simplicial complex $P_{d}(\Gamma)$. Assume that the Bott map $\beta_{L,*}$ is an isomorphism on the $(k-1)$-skeleton of $P_{d}(\Gamma)$. Using the Mayer--Vietoris sequence, we decompose the $k$-skeleton into the ``old'' $(k-1)$-skeleton and the ``new'' part consisting of $k$-simplices. For the old part and its intersection with the new part (which is itself of dimension $k-1$), $\beta_{L,*}$ is an isomorphism by the inductive hypothesis.

For the new part, we consider the disjoint union of the interiors of all $k$-simplices. Since $\Gamma$ acts properly and cocompactly, there are only finitely many $\Gamma$-orbits of simplices. For each $k$-simplex $\sigma$, its stabilizer $\Gamma_{\sigma}$ is a finite subgroup, and its interior is $\Gamma_{\sigma}$-equivariantly strongly Lipschitz homotopy equivalent to its barycenter $x_{\sigma}$. Because the $K$-theory of localization algebras is invariant under such equivalences, the problem reduces to showing that $\beta_{L,*}$ is an isomorphism for orbits of finite subgroups. By the descent principle, this identifies with the map
$$
K_{*}(C^{*}_{L}(\mathrm{pt}, C(W)\hat{\otimes}B)^{\Gamma_\sigma}\hat{\otimes}\mathcal{S}) \to K_{*}(C^{*}_{L}(\mathrm{pt}, C(W)\hat{\otimes}\mathcal{A}(\mathcal{H})\hat{\otimes}B)^{\Gamma_\sigma}).
$$
Since $W$ is $\Gamma_\sigma$-contractible for finite $\Gamma_\sigma \le \Gamma$, this map is an isomorphism by \cite{higson-kasparov-trout-1998}. Finally, the Five Lemma implies that the Bott map is an isomorphism on the entire $k$-skeleton, completing the induction.
\end{proof}

The inclusion of constants $c: \mathbb{C} \to C(W)$ induces a homomorphism $\tilde{c}_{L}: C^{*}_{L}(P_{d}(\Gamma),B)^{\Gamma} \to C^{*}_{L}(P_{d}(\Gamma),C(W)\hat{\otimes}B)^{\Gamma}$. This map also yields a $K$-theoretic isomorphism:
\begin{lemma}\label{C(W) iso}
  For each $d>0$, the map
  $$
  \tilde{c}_{L,*}:K_{*}(C^{*}_{L}(P_{d}(\Gamma),B)^{\Gamma}) \to K_{*}(C^{*}_{L}(P_{d}(\Gamma),C(W)\hat{\otimes}B)^{\Gamma})
  $$
  is an isomorphism.
\end{lemma}
\begin{proof}[Sketch of proof]

Cover $P_d(\Gamma)$ by $\{\Gamma\cdot V_i\}$ where $V_i$ is $F_i$-contractible. By Mayer--Vietoris and the descent principle, we reduce to the map $K_{*}(B\rtimes _{r}F_{i}) \to K_{*}((C(W)\hat{\otimes}B)\rtimes _{r}F_{i})$. Since $W$ is $F_{i}$-contractible, $c$ is an $F_{i}$-equivariant homotopy equivalence, which induces the isomorphism.
\end{proof}

\begin{proof}[Proof of Theorem \ref{strongNovikov}]

  To establish the injectivity of the assembly map $\mu_{\Gamma}$, we construct the following commutative diagram. Here, the vertical arrows are induced by the coefficient inclusion $\tilde{c}$ (defined in Lemma \ref{C(W) iso}) and the Bott periodicity maps $\beta$ (defined in \eqref{bott-cw-ah-L}), respectively.
\begin{equation}\label{strong Novikov cmmu}
  \begin{tikzcd}[
  column sep=small,
  every cell/.append style={font=\footnotesize}, 
  every label/.append style={font=\scriptsize} 
  ]
	{\lim\limits_{d\to\infty}K_{*}(C_{L}^{*}(P_{d}(\Gamma),B)^{\Gamma})} && {\lim\limits_{d\to\infty}K_{*}(C^{*}(P_{d}(\Gamma),B)^{\Gamma})} \\
	\\
	{\lim\limits_{d\to\infty}K_{*}(C_{L}^{*}(P_{d}(\Gamma),C(W)\hat{\otimes}B)^{\Gamma})} && {\lim\limits_{d\to\infty}K_{*}(C^{*}(P_{d}(\Gamma),C(W)\hat{\otimes}B)^{\Gamma})} \\
	{\lim\limits_{d\to\infty}K_{*}(C_{L}^{*}(P_{d}(\Gamma),C(W)\hat{\otimes}B)^{\Gamma}\hat{\otimes}\mathcal{S})} && {\lim\limits_{d\to\infty}K_{*}(C^{*}(P_{d}(\Gamma),C(W)\hat{\otimes}B)^{\Gamma}\hat{\otimes}\mathcal{S})} \\
	\\
	{\lim\limits_{d\to\infty}K_{*}(C_{L}^{*}(P_{d}(\Gamma),C(W)\hat{\otimes}\mathcal{A}(H)\hat{\otimes}B)^{\Gamma})} && {\lim\limits_{d\to\infty}K_{*}(C^{*}(P_{d}(\Gamma),C(W)\hat{\otimes}\mathcal{A}(H)\hat{\otimes}B)^{\Gamma})}
	\arrow["{\mu_{\Gamma}}", from=1-1, to=1-3]
	\arrow["\tilde{c}_{L,*}"', from=1-1, to=3-1]
	\arrow[from=1-3, to=3-3]
	\arrow[from=3-1, to=3-3]
	\arrow[from=4-1, to=4-3]
	\arrow["{\beta_{L,*}}"', from=4-1, to=6-1]
	\arrow[from=4-3, to=6-3]
	\arrow["{{\mu_{ \mathcal{A},W}^{\Gamma}}}", from=6-1, to=6-3]
\end{tikzcd}
\end{equation}
By Lemma \ref{isomorphism with W}, the bottom horizontal map $\mu_{\mathcal{A},W}^{\Gamma}$ is an isomorphism. Since the Bott map $\beta_{L,*}$ is an isomorphism (Lemma \ref{localization iso}), the middle horizontal maps are injective. Finally, as the coefficient map $\tilde{c}_{L,*}$ is an isomorphism by Lemma \ref{C(W) iso}, we conclude that the assembly map $\mu_{\Gamma}$ is injective.
\end{proof}

\subsection{Application to the coarse Baum--Connes conjecture}

In \cite[Theorem 5.3]{Clement-2023}, Dell'Aiera showed that $\Gamma$ satisfies the Novikov conjecture if both $\Lambda$ and $X = \Gamma/\Lambda$ admit coarse embeddings into Hilbert spaces. This conclusion follows as a corollary to his main result on the Baum--Connes conjecture \cite[Theorem 4.1]{Clement-2023}, which requires the stronger condition of $\Gamma$-equivariant coarse embeddability for $X$. A similar result can be obtained by applying the coarse fibration techniques developed in \cite{deng-guo-2025}.

We claim that the natural quotient map $\pi: \Gamma \to X = \Gamma/\Lambda$ defines a coarse $\Lambda$-fibration structure on $\Gamma$ over the base space $X$.

First, for any $\gamma_1, \gamma_2 \in \Gamma$, as defined in \eqref{X-metric},
$$
d(\pi(\gamma_1), \pi(\gamma_2))=d(\gamma_1 \Lambda, \gamma_2 \Lambda)  = \inf_{\lambda, \lambda' \in \Lambda} |\lambda \gamma_1^{-1} \gamma_2 \lambda'| \le |\gamma_1^{-1} \gamma_2|.
$$ 
Thus, $\pi$ is bornologous.

Second, for any $x\in X$, the the fiber $\pi^{-1}(x)$ is a left coset $\gamma \Lambda$ and is naturally isometric to $\Lambda$. Thus, all the fibers are uniformly coarsely equivalent to $\Lambda$.

Finally, for any $R>0$ and $x \in X$ with $\pi^{-1}(x)=\gamma \Lambda$, the preimage $\pi^{-1}(B_X(x, R))$ consists of all $\gamma' \in \Gamma$ such that $d(\pi(\gamma'), x) \le R$. Rewrite the equation \eqref{X-metric} as 
$$
d(\gamma_1 \Lambda, \gamma_2 \Lambda)  = \inf_{\lambda, \lambda' \in \Lambda} |\lambda \gamma_1^{-1} \gamma_2 \lambda'|=\inf_{\lambda \in \Lambda} |\gamma_1^{-1} \gamma_2 \lambda|=\inf_{\lambda \in \Lambda}d_{\Gamma}(\gamma_{1},\gamma_{2}\lambda).
$$
This means the set $\pi^{-1}(B_X(x, R))$ is exactly the $R$-neighborhood of the coset $\gamma \Lambda$ in $\Gamma$, and is coarse equivalent to $\Lambda$.

Consequently, the sequence $\Lambda \hookrightarrow \Gamma \overset{\pi}{\rightarrow } X$ is a coarse fibration construction. Moreover, assume that $\Lambda$ and $X$ admit coarse embeddings into Hilbert space. In the terminology of Corollary 5.1 in \cite{deng-guo-2025}, $\Gamma$ admits a CE-by-CE coarse fibration structure. Therefore, the twisted coarse Baum--Connes conjecture with coefficients holds for $\Gamma$. This implies the usual coarse Baum--Connes conjecture, and thus we have the following theorem. 
\begin{theorem}\label{coarseBCC}
    Let $(\Gamma, \Lambda)$ be a Hecke pair of discrete groups. If both $\Lambda$ and $X$ admit coarse embeddings into Hilbert space, then $\Gamma$ satisfies the coarse Baum--Connes conjecture and consequently the Novikov conjecture. \qed
\end{theorem}

\section*{Acknowledgements}
The authors would like to express sincere gratitude to Kang Li and Jintao Deng for helpful discussions and suggestions.
Liang Guo is partially supported by the Chinese Postdoctoral Science Foundation (No. 2025M773059) and the Research Start-up Fund of the Shanghai Institute for Mathematics and Interdisciplinary Sciences (SIMIS).
Hang Wang is supported by the grants 23JC1401900, NSFC 12271165 and in part by Science and Technology Commission of Shanghai Municipality (No. 22DZ2229014).

\bibliographystyle{amsalpha}
\bibliography{references}

@article{Cuntz-1983,
  AUTHOR = {Cuntz, Joachim},
     TITLE = {{$K$}-theoretic amenability for discrete groups},
   JOURNAL = {J. Reine Angew. Math.},
  FJOURNAL = {Journal f{\"}ur die Reine und Angewandte Mathematik. [Crelle's
              Journal]},
    VOLUME = {344},
      YEAR = {1983},
     PAGES = {180--195},
      ISSN = {0075-4102,1435-5345},
   MRCLASS = {46L80 (19K99)},
  MRNUMBER = {716254},
MRREVIEWER = {Autorreferat},
       DOI = {10.1515/crll.1983.344.180},
       URL = {https://doi.org/10.1515/crll.1983.344.180},
}

@article{Julg-Valette-1985,
  AUTHOR = {Julg, Pierre and Valette, Alain},
     TITLE = {{$K$}-theoretic amenability for {${\rm SL}\sb{2}({\bf
              Q}\sb{p})$}, and the action on the associated tree},
   JOURNAL = {J. Funct. Anal.},
  FJOURNAL = {Journal of Functional Analysis},
    VOLUME = {58},
      YEAR = {1984},
    NUMBER = {2},
     PAGES = {194--215},
      ISSN = {0022-1236},
   MRCLASS = {22E50 (18F25 19K99 46L80 46M20 58G12)},
  MRNUMBER = {757995},
MRREVIEWER = {Alan\ L. T. Paterson},
       DOI = {10.1016/0022-1236(84)90039-9},
       URL = {https://doi.org/10.1016/0022-1236(84)90039-9},
}

@article{Kasparov-1988,
  author  = {Kasparov, Gennadi},
  TITLE = {Equivariant {$KK$}-theory and the {N}ovikov conjecture},
   JOURNAL = {Invent. Math.},
  FJOURNAL = {Inventiones Mathematicae},
    VOLUME = {91},
      YEAR = {1988},
    NUMBER = {1},
     PAGES = {147--201},
      ISSN = {0020-9910,1432-1297},
   MRCLASS = {58G12 (19K33 19K56 46L80 46M20 53C20 57R67)},
  MRNUMBER = {918241},
MRREVIEWER = {Jonathan\ M.\ Rosenberg},
       DOI = {10.1007/BF01404917},
       URL = {https://doi.org/10.1007/BF01404917},
}

@article{Higson-Kasparov-2001,
  AUTHOR = {Higson, Nigel and Kasparov, Gennadi},
     TITLE = {{$E$}-theory and {$KK$}-theory for groups which act properly and isometrically on {H}ilbert space},
   JOURNAL = {Invent. Math.},
  FJOURNAL = {Inventiones Mathematicae},
    VOLUME = {144},
      YEAR = {2001},
    NUMBER = {1},
     PAGES = {23--74},
      ISSN = {0020-9910,1432-1297},
   MRCLASS = {19K35 (19L47 46L80)},
  MRNUMBER = {1821144},
MRREVIEWER = {Emmanuel\ C.\ Germain},
       DOI = {10.1007/s002220000118},
       URL = {https://doi.org/10.1007/s002220000118},
}

@article{Yu-2000,
  author  = {Yu, Guoliang},
  AUTHOR = {Yu, Guoliang},
     TITLE = {The coarse {B}aum-{C}onnes conjecture for spaces which admit a uniform embedding into {H}ilbert space},
   JOURNAL = {Invent. Math.},
  FJOURNAL = {Inventiones Mathematicae},
    VOLUME = {139},
      YEAR = {2000},
    NUMBER = {1},
     PAGES = {201--240},
      ISSN = {0020-9910,1432-1297},
   MRCLASS = {19K56 (46L80 57R67 58J22)},
  MRNUMBER = {1728880},
MRREVIEWER = {Nigel\ Higson},
       DOI = {10.1007/s002229900032},
       URL = {https://doi.org/10.1007/s002229900032},
}

@article{Kasparov-Yu-2012,
  AUTHOR = {Kasparov, Gennadi and Yu, Guoliang},
     TITLE = {The {N}ovikov conjecture and geometry of {B}anach spaces},
   JOURNAL = {Geom. Topol.},
  FJOURNAL = {Geometry \& Topology},
    VOLUME = {16},
      YEAR = {2012},
    NUMBER = {3},
     PAGES = {1859--1880},
      ISSN = {1465-3060,1364-0380},
   MRCLASS = {19K56 (46L80)},
  MRNUMBER = {2980001},
MRREVIEWER = {Otgonbayar\ Uuye},
       DOI = {10.2140/gt.2012.16.1859},
       URL = {https://doi.org/10.2140/gt.2012.16.1859},
}

@article{Lafforgue-2002,
  AUTHOR = {Lafforgue, Vincent},
     TITLE = {{$K$}-th{\'e}orie bivariante pour les alg{\`e}bres de {B}anach et conjecture de {B}aum-{C}onnes},
   JOURNAL = {Invent. Math.},
  FJOURNAL = {Inventiones Mathematicae},
    VOLUME = {149},
      YEAR = {2002},
    NUMBER = {1},
     PAGES = {1--95},
      ISSN = {0020-9910,1432-1297},
   MRCLASS = {19K35 (46H25 46L80 46M20 58J22)},
  MRNUMBER = {1914617},
MRREVIEWER = {Georges\ Skandalis},
       DOI = {10.1007/s002220200213},
       URL = {https://doi.org/10.1007/s002220200213},
}

@article{Lafforgue-2012,
    AUTHOR = {Lafforgue, Vincent},
     TITLE = {La conjecture de {B}aum-{C}onnes {\`a} coefficients pour les groupes hyperboliques},
   JOURNAL = {J. Noncommut. Geom.},
  FJOURNAL = {Journal of Noncommutative Geometry},
    VOLUME = {6},
      YEAR = {2012},
    NUMBER = {1},
     PAGES = {1--197},
      ISSN = {1661-6952,1661-6960},
   MRCLASS = {19K35 (19L47 20F67 46L80)},
  MRNUMBER = {2874956},
MRREVIEWER = {Jean-Louis\ Tu},
       DOI = {10.4171/JNCG/89},
       URL = {https://doi.org/10.4171/JNCG/89},
}

@article{KasparovSkandalisbolic,
   AUTHOR = {Kasparov, Gennadi and Skandalis, Georges},
     TITLE = {Groups acting properly on ``bolic'' spaces and the {N}ovikov conjecture},
   JOURNAL = {Ann. of Math. (2)},
  FJOURNAL = {Annals of Mathematics. Second Series},
    VOLUME = {158},
      YEAR = {2003},
    NUMBER = {1},
     PAGES = {165--206},
      ISSN = {0003-486X,1939-8980},
   MRCLASS = {58J22 (19K35 19M05 57R65)},
  MRNUMBER = {1998480},
MRREVIEWER = {Tsuyoshi\ Kato},
       DOI = {10.4007/annals.2003.158.165},
       URL = {https://doi.org/10.4007/annals.2003.158.165},
}

@article{Brodzki-Guentner-Higson-2019,
  AUTHOR = {Brodzki, Jacek and Guentner, Erik and Higson, Nigel},
  TITLE = {A differential complex for {${\rm CAT}(0)$} cubical spaces},
   JOURNAL = {Adv. Math.},
  FJOURNAL = {Advances in Mathematics},
    VOLUME = {347},
      YEAR = {2019},
     PAGES = {1054--1111},
      ISSN = {0001-8708,1090-2082},
   MRCLASS = {46L80 (20E08 20F65 46L07)},
  MRNUMBER = {3923391},
MRREVIEWER = {Ignacio\ Vergara},
       DOI = {10.1016/j.aim.2019.03.009},
       URL = {https://doi.org/10.1016/j.aim.2019.03.009},
}

@article{Oyono-2001,
  author  = {Oyono-Oyono, Herv{\'e}},
  TITLE = {{B}aum-{C}onnes conjecture and group actions on trees},
   JOURNAL = {$K$-Theory},
  FJOURNAL = {$K$-Theory. An Interdisciplinary Journal for the Development, Application, and Influence of $K$-Theory in the Mathematical Sciences},
    VOLUME = {24},
      YEAR = {2001},
    NUMBER = {2},
     PAGES = {115--134},
      ISSN = {0920-3036,1573-0514},
   MRCLASS = {19K35 (20E08)},
  MRNUMBER = {1869625},
MRREVIEWER = {Michel\ Matthey},
       DOI = {10.1023/A:1012786413219},
       URL = {https://doi.org/10.1023/A:1012786413219},
}

@article{MR4422220,
  author  = {Deng, Jintao},
  TITLE = {The {N}ovikov conjecture and extensions of coarsely embeddable groups},
   JOURNAL = {J. Noncommut. Geom.},
  FJOURNAL = {Journal of Noncommutative Geometry},
    VOLUME = {16},
      YEAR = {2022},
    NUMBER = {1},
     PAGES = {265--310},
      ISSN = {1661-6952,1661-6960},
   MRCLASS = {46L80 (19K56)},
  MRNUMBER = {4422220},
MRREVIEWER = {Sanaz\ Pooya},
       DOI = {10.4171/jncg/437},
       URL = {https://doi.org/10.4171/jncg/437},
}

@article{Clement-2023,
  author  = {Dell'Aiera, Cl{\'e}ment},
     TITLE = {Coarse geometry of {H}ecke pairs and the {B}aum-{C}onnes
              conjecture},
   JOURNAL = {Pacific J. Math.},
  FJOURNAL = {Pacific Journal of Mathematics},
    VOLUME = {322},
      YEAR = {2023},
    NUMBER = {1},
     PAGES = {21--37},
      ISSN = {0030-8730,1945-5844},
   MRCLASS = {19K35 (22D55 46L80 55N20)},
  MRNUMBER = {4582953},
MRREVIEWER = {George\ A.\ Willis},
       DOI = {10.2140/pjm.2023.322.21},
       URL = {https://doi.org/10.2140/pjm.2023.322.21},
}

@article{higson-kasparov-trout-1998,
  AUTHOR = {Higson, Nigel and Kasparov, Gennadi and Trout, Jody},
     TITLE = {A {B}ott periodicity theorem for infinite-dimensional {E}uclidean space},
   JOURNAL = {Adv. Math.},
  FJOURNAL = {Advances in Mathematics},
    VOLUME = {135},
      YEAR = {1998},
    NUMBER = {1},
     PAGES = {1--40},
      ISSN = {0001-8708,1090-2082},
   MRCLASS = {19K35 (19L47 46L80 57R67 58G12)},
  MRNUMBER = {1617411},
MRREVIEWER = {Evgeniy\ V.\ Troitski\u i},
       DOI = {10.1006/aima.1997.1706},
       URL = {https://doi.org/10.1006/aima.1997.1706},
}

@article{gong-wang-yu-maxroe,
  AUTHOR = {Gong, Guihua and Wang, Qin and Yu, Guoliang},
     TITLE = {Geometrization of the strong {N}ovikov conjecture for
              residually finite groups},
   JOURNAL = {J. Reine Angew. Math.},
  FJOURNAL = {Journal f\"ur die Reine und Angewandte Mathematik. [Crelle's Journal]},
    VOLUME = {621},
      YEAR = {2008},
     PAGES = {159--189},
      ISSN = {0075-4102,1435-5345},
   MRCLASS = {58J22 (22E40 57R67)},
  MRNUMBER = {2431253},
MRREVIEWER = {Jean-Louis\ Tu},
       DOI = {10.1515/CRELLE.2008.061},
       URL = {https://doi.org/10.1515/CRELLE.2008.061},
}

@book{willett2020higher,
  AUTHOR = {Willett, Rufus and Yu, Guoliang},
     TITLE = {Higher index theory},
    SERIES = {Cambridge Studies in Advanced Mathematics},
    VOLUME = {189},
 PUBLISHER = {Cambridge University Press, Cambridge},
      YEAR = {2020},
     PAGES = {xi+581},
      ISBN = {978-1-108-49106-8},
   MRCLASS = {46-02 (19Kxx 46L80 47L60 58J22)},
  MRNUMBER = {4411373},
MRREVIEWER = {Man-Ho\ Ho},
       DOI = {10.1017/9781108867351},
       URL = {https://doi.org/10.1017/9781108867351},
}

@book{Williams2007,
  author    = {Williams, Dana P.},
  TITLE = {Crossed products of {$C{^\ast}$}-algebras},
    SERIES = {Mathematical Surveys and Monographs},
    VOLUME = {134},
 PUBLISHER = {American Mathematical Society, Providence, RI},
      YEAR = {2007},
     PAGES = {xvi+528},
      ISBN = {978-0-8218-4242-3; 0-8218-4242-0},
   MRCLASS = {46-02 (22D25 46L05 46L35 46L55 46L85)},
  MRNUMBER = {2288954},
MRREVIEWER = {Jonathan\ M.\ Rosenberg},
       DOI = {10.1090/surv/134},
       URL = {https://doi.org/10.1090/surv/134},
}

@article{Yu-1995,
   AUTHOR = {Yu, Guo Liang},
     TITLE = {Coarse {B}aum-{C}onnes conjecture},
   JOURNAL = {$K$-Theory},
  FJOURNAL = {$K$-Theory. An Interdisciplinary Journal for the Development, Application, and Influence of $K$-Theory in the Mathematical Sciences},
    VOLUME = {9},
      YEAR = {1995},
    NUMBER = {3},
     PAGES = {199--221},
      ISSN = {0920-3036,1573-0514},
   MRCLASS = {58G12 (19K35 20F32 46L85)},
  MRNUMBER = {1344138},
MRREVIEWER = {John\ Roe},
       DOI = {10.1007/BF00961664},
       URL = {https://doi.org/10.1007/BF00961664},
}

@article{MR2015025,
  AUTHOR = {Tzanev, Kroum},
     TITLE = {Hecke {$C^*$}-algebras and amenability},
   JOURNAL = {J. Operator Theory},
  FJOURNAL = {Journal of Operator Theory},
    VOLUME = {50},
      YEAR = {2003},
    NUMBER = {1},
     PAGES = {169--178},
      ISSN = {0379-4024,1841-7744},
   MRCLASS = {46L55 (22D25 43A20)},
  MRNUMBER = {2015025},
MRREVIEWER = {Nadia\ S.\ Larsen},
}

@article{MR2465930,
  author  = {Kaliszewski, Steven. and Landstad, Magnus B. and Quigg, John},
  TITLE = {Hecke {$C^*$}-algebras, {S}chlichting completions and {M}orita equivalence},
   JOURNAL = {Proc. Edinb. Math. Soc. (2)},
  FJOURNAL = {Proceedings of the Edinburgh Mathematical Society. Series II},
    VOLUME = {51},
      YEAR = {2008},
    NUMBER = {3},
     PAGES = {657--695},
      ISSN = {0013-0915,1464-3839},
   MRCLASS = {46L05 (20C08 46L55)},
  MRNUMBER = {2465930},
MRREVIEWER = {\'Eric\ Leichtnam},
       DOI = {10.1017/S0013091506001419},
       URL = {https://doi.org/10.1017/S0013091506001419},
}

@article{fu-wang-2016,
  AUTHOR = {Fu, Benyin and Wang, Xianjin},
     TITLE = {The equivariant coarse {B}aum-{C}onnes conjecture for spaces which admit an equivariant coarse embedding into {H}ilbert
              space},
   JOURNAL = {J. Funct. Anal.},
  FJOURNAL = {Journal of Functional Analysis},
    VOLUME = {271},
      YEAR = {2016},
    NUMBER = {4},
     PAGES = {799--832},
      ISSN = {0022-1236,1096-0783},
   MRCLASS = {58B34 (46B85 46L85)},
  MRNUMBER = {3507991},
MRREVIEWER = {Cristian\ Ivanescu},
       DOI = {10.1016/j.jfa.2016.05.005},
       URL = {https://doi.org/10.1016/j.jfa.2016.05.005},
}

@misc{deng-guo-2025,
  title={Twisted {R}oe algebras and their {$K$}-theory}, 
  note={Preprint, arXiv:2409.16556},
      author={Jintao Deng and Liang Guo},
      year={2025},
      eprint={2409.16556},
      archivePrefix={arXiv},
      primaryClass={math.KT},
      url={https://arxiv.org/abs/2409.16556}, 
}

@article{Tu-groupoid,
  AUTHOR = {Tu, Jean-Louis},
     TITLE = {La conjecture de {B}aum-{C}onnes pour les feuilletages
              moyennables},
   JOURNAL = {$K$-Theory},
  FJOURNAL = {$K$-Theory. An Interdisciplinary Journal for the Development, Application, and Influence of $K$-Theory in the Mathematical Sciences},
    VOLUME = {17},
      YEAR = {1999},
    NUMBER = {3},
     PAGES = {215--264},
      ISSN = {0920-3036,1573-0514},
   MRCLASS = {19K35 (46L85 58J22)},
  MRNUMBER = {1703305},
MRREVIEWER = {Hiroshi\ Takai},
       DOI = {10.1023/A:1007744304422},
       URL = {https://doi.org/10.1023/A:1007744304422},
}

@article {Guentner-Higson-Trout-2000,
    AUTHOR = {Guentner, Erik and Higson, Nigel and Trout, Jody},
     TITLE = {Equivariant {$E$}-theory for {$C^*$}-algebras},
   JOURNAL = {Mem. Amer. Math. Soc.},
  FJOURNAL = {Memoirs of the American Mathematical Society},
    VOLUME = {148},
      YEAR = {2000},
    NUMBER = {703},
     PAGES = {viii+86},
      ISSN = {0065-9266,1947-6221},
   MRCLASS = {46L80 (19K35 22D25 46L85)},
  MRNUMBER = {1711324},
MRREVIEWER = {Vicumpriya\ S.\ Perera},
       DOI = {10.1090/memo/0703},
       URL = {https://doi.org/10.1090/memo/0703},
}

@article {Tu-1999,
    AUTHOR = {Tu, Jean-Louis},
     TITLE = {The {B}aum-{C}onnes conjecture and discrete group actions on trees},
   JOURNAL = {$K$-Theory},
  FJOURNAL = {$K$-Theory. An Interdisciplinary Journal for the Development, Application, and Influence of $K$-Theory in the Mathematical Sciences},
    VOLUME = {17},
      YEAR = {1999},
    NUMBER = {4},
     PAGES = {303--318},
      ISSN = {0920-3036,1573-0514},
   MRCLASS = {19K35 (46L80)},
  MRNUMBER = {1706113},
MRREVIEWER = {Yuri\ A.\ Kordyukov},
       DOI = {10.1023/A:1007751625568},
       URL = {https://doi.org/10.1023/A:1007751625568},
}

@book {CCJJV-2001,
    AUTHOR = {Cherix, Pierre-Alain and Cowling, Michael and Jolissaint, Paul and Julg, Pierre and Valette, Alain},
     TITLE = {Groups with the {H}aagerup property},
    SERIES = {Progress in Mathematics},
    VOLUME = {197},
      NOTE = {Gromov's a-T-menability},
 PUBLISHER = {Birkh\"auser Verlag, Basel},
      YEAR = {2001},
     PAGES = {viii+126},
      ISBN = {3-7643-6598-6},
   MRCLASS = {22D10 (22-02 22D25 22E30 43A07 46Lxx)},
  MRNUMBER = {1852148},
MRREVIEWER = {Tullio\ G.\ Ceccherini-Silberstein},
       DOI = {10.1007/978-3-0348-8237-8},
       URL = {https://doi.org/10.1007/978-3-0348-8237-8},
}

@book {BHV-2008,
    AUTHOR = {Bekka, Bachir and de la Harpe, Pierre and Valette, Alain},
     TITLE = {Kazhdan's property ({T})},
    SERIES = {New Mathematical Monographs},
    VOLUME = {11},
 PUBLISHER = {Cambridge University Press, Cambridge},
      YEAR = {2008},
     PAGES = {xiv+472},
      ISBN = {978-0-521-88720-5},
   MRCLASS = {22-02 (22E40 28D15 37A15 43A07 43A35)},
  MRNUMBER = {2415834},
MRREVIEWER = {Markus\ Neuhauser},
       DOI = {10.1017/CBO9780511542749},
       URL = {https://doi.org/10.1017/CBO9780511542749},
}

\vspace{5mm}

\noindent(L.~Guo) Shanghai Institute for Mathematics and Interdisciplinary Sciences (SIMIS), Shanghai, 200433, P.~R.~China.\\ Research Institute of Intelligent Complex Systems, Fudan University, Shanghai, 200433, China\\
Email: liangguo@simis.cn
\vspace{.2cm}

\noindent(H.~Wang) 
Research Center for Operator Algebras, East China Normal University, Shanghai 200062, China \\ 
Email: wanghang@math.ecnu.edu.cn 
\vspace{.2cm}

\noindent(X.~Yao) School of Mathematical Sciences, East China Normal University (ECNU), Shanghai, 200241, China. \\ Research Center for Operator Algebras (RCOA), East China Normal University (ECNU), Shanghai, 200062, China \\
Email: yaoze@whut.edu.cn

\end{document}